\documentclass[hidelinks,onefignum,onetabnum]{siamart220329}
\usepackage{graphicx}
\usepackage{amssymb}
\usepackage{amsmath}
\usepackage{mathrsfs}
\usepackage{enumitem}
\newcommand{\grad}{\operatorname{grad}}

\newcommand{\supp}{\operatorname{supp}} % supp
\newcommand{\conv}{\operatorname{conv}} % Convex
\newcommand{\Hb}{\mathcal{H}} % Hilbert transform
\newcommand{\R}{\mathbb{R}} % Real number
\newcommand{\RBV}{\mathcal{R}} % Radon-BV operator
\newcommand{\Radon}{\mathscr{R}} % Radon transform
\newcommand{\Rnorm}{\mathrm{R}} % R-norm
\newcommand{\di}{\mathrm{d}} % dx, du, dv ...
\newcommand{\Sf}{\mathbb{S}} % surface of unit ball
\newcommand{\Sw}{\mathcal{S}} % Schwartz space
\newcommand{\F}{\mathcal{F}} % Radon-BV space
\newcommand{\Ft}{\mathscr{F}} % Fourier transform
\newcommand{\M}{\mathcal{M}} % Radon measure
\newcommand{\D}{\mathcal{D}} % test functions
\newcommand{\e}{\mathrm{e}} %
\newcommand{\repu}{\mathrm{RePU}} % RePU
\newcommand{\relu}{\mathrm{ReLU}} % ReLU
\newcommand{\PP}{\mathbb{P}} % variation space dictionary
\newcommand{\FF}{\mathbb{F}} % variation space dictionary
\newcommand{\EE}{\mathbb{E}} % expectation
\newcommand{\ext}{\mathrm{ext}}
\newcommand{\K}{\mathcal{K}} % variation space
\newcommand{\N}{\mathbb{N}} % 0,1,2,...
\newcommand{\diam}{\mathrm{diam}} % Diameter
\newcommand{\B}{\mathbb{B}} % a new dictionary for proof of variation place and extended Barron space
\newcommand{\Poly}{\mathscr{P}} % polynomial
\DeclareMathOperator{\esssup}{ess\,sup} % essential supremum
\newcommand{\lr}[1]{\left( #1 \right)}
\newcommand{\abs}[1]{\left\lvert #1 \right\rvert}
\newcommand{\norm}[2]{\| #1\|_{#2}}

\newsiamremark{remark}{Remark}
\headers{Function, derivative approximation by shallow neural networks}{Yuanyuan Li and Shuai Lu}

\title{Function and derivative approximation by shallow neural networks\thanks{Submitted to the editors DATE.
\funding{This work is supported by National Key Research and Development Programs of China (No. 2023YFA1009103), NSFC (No. 11925104) and Science and Technology Commission of Shanghai Municipality (No. 23JC1400501)}}}

\author{Yuanyuan Li\thanks{School of Mathematics Sciences, Fudan University, Shanghai, China
  (\email{liyuanyuan20@fudan.edu.cn}, \email{slu@fudan.edu.cn}). Corresponding author: S. Lu.}
\and Shuai Lu\footnotemark[2]
}

\ifpdf
\hypersetup{
  pdftitle={Function and derivative approximation by shallow neural networks},
  pdfauthor={Yuanyuan Li, and Shuai Lu}
}
\fi

\begin{document}

\maketitle
\begin{abstract}
We investigate a Tikhonov regularization scheme specifically tailored for shallow neural networks within the context of solving a classic problem: approximating an unknown function and its derivatives in a unit cubic domain based on noisy measurements.
The proposed Tikhonov regularization scheme incorporates a penalty term that takes three distinct yet intricately related network (semi)norms: the extended Barron norm, the variation norm, and the Radon-BV seminorm. These choices of the penalty term are contingent upon the specific architecture of the neural network being utilized. We establish the connection between various network norms and particularly trace the dependence of the dimensionality index, aiming to deepen our understanding of how these norms interplay with each other.
We revisit the universality of function approximation through various norms, establish rigorous error-bound analysis for the Tikhonov regularization scheme, and explicitly elucidate the dependency of the dimensionality index, providing a clearer understanding of how the dimensionality affects the approximation performance and how one designs a neural network with diverse approximating tasks.

%\begin{itemize}
%    \item 65J22: Inverse problems;
%    \item 65J20: Improperly posed problems; regularization;
%    \item 65-04: Explicit machine computation and programs (not the theory of computation or programming);
%    \item 65B05: Extrapolation to the limit, deferred corrections;
%    \item 65D15: Algorithms for functional approximation
%\end{itemize}}
\end{abstract}

%\tableofcontents

\section{Introduction}\label{se_intro}
%%%%%%%%%%%%%%%
% 1. Describe the problem
% 2. Overview of the UAT for functions and derivatives
% 3. Overview of numerical differentiation
% 4. Our contribution and outline
%%%%%%%%%%%%%%%

This paper delves into the classical problem of approximating an unknown smooth function and its derivatives within a unit cubic domain, specifically $\Omega = (0,1)^d$, in the context of noisy measurement data. Distinguishing from existing research, our analysis focuses on the scenario where the approximating function is a two-layer neural network, commonly referred to as a shallow neural network. This neural network is articulated as follows:
\begin{align}\label{eq:shallow}
    f_n(x) = \sum_{i=1}^{n}a_i\sigma(w_i\cdot x + b_i) + a_0.
\end{align}
Within the above architecture, $w_i\in \R^d$ and $b_i\in\R$ respectively represent the weights and biases of the hidden layer. Similarly,
$a_i\in \R$ and $a_0\in \R$ serve as the weights and bias of the output layer.
The function $\sigma(\cdot)$ in
(\ref{eq:shallow}) is known as the activation function. It plays a crucial role in introducing non-linearity into the network, enabling it to learn and approximate complex functions. Common activation functions include Rectified Linear Unit (ReLU), sigmoid, tanh and so on.
Specifically, we focus on the Rectified Power Unit ($\repu$) $\sigma_{k}(z) = \left(\max\{0,z\}\right)^k$ which is also known as the ReLU$^k$ function. It is noteworthy that when $k$ is set to $1$, $\sigma_1(z)$ corresponds to the standard $\relu$ function. We note that there has been extensive research exploring the functional spaces of two-layer neural networks, specifically those employing the $\relu$ activation function, as exemplified in \cite{SESS19, OSWS20, EMW22}, as well as those utilizing the $\repu$ activation function, such as in \cite{Xu20, PN21, abdeljawad_integral_2022, LLMP23}.

Approximating a function based on its measurements has been a long-standing task in numerical analysis and approximation theory, playing an important role in supervised learning, as outlined in \cite{Aronszajn50,  Wahba90, DeVore1998}.
Universal approximation theory (UAT) has been firmly established such that a two-layer neural network, equipped with the right parameters, can approximate any continuous function with arbitrary precision \cite{Cybenko89, Hornik91, Barron93, Yarotsky2017, GKNV2022}. Systematic studies on the universality properties of two-layer neural networks, specifically those employing the ReLU activation function, have been conducted in \cite{EMW22, SX22}, considering various {\it a priori} knowledge about the unknown target function.
Regarding UAT for the derivatives, a thorough and systematic investigation has been conducted in \cite{SX22, SX22sharp, SX23}, where the approximation rate about the number of neurons has been comprehensively established.
Some attempts have been undertaken to approximate unknown functions using deep neural networks with $\relu$ activation functions. In particular, we cite \cite{Ali_Nouy_21, MM22} for using $\relu$ and $\repu$ neural networks to approximate functions related to Besov spaces. Additionally, \cite{zhou_universality_2020} demonstrates the universality of deep convolutional neural networks when employing the ReLU activation function. Furthermore, \cite{CPF} proposes an approach for approximating classification functions, while \cite{LSYZ21} analyzes the optimal approximation of smooth functions, considering both the width and depth of the networks simultaneously. %Notably, \cite{PRV21} emphasizes the importance of considering the undesirable topological properties of the function set when utilizing neural networks for function approximation.

However, it is noteworthy that the authors above have not considered the presence of noisy data in their investigations.
Including noisy data in such approximations is crucial for real-world applications, where measurements are often imperfect and contaminated with noise.
Addressing this challenge, which involves developing robust approximation methods that can effectively handle noisy data, remains an important direction for future research \cite{JMFU17, Baoetal20, DSY_21}.  These noise models typically show as
\begin{align}\label{eq:noisymeasurement}
  \|f-f^{\delta}\|_{L^2(\Omega)} \leq \delta
\end{align}
where $f$ is the unknown exact solution, $f^{\delta}$ is the noisy measurement and small $\delta>0$ represents the noise level.
%{\color{red}\sout{While it is well understood that approximating an unknown function is stable under noise propagation,}
The derivative approximation of an unknown function poses inherent challenges due to its ill-posed nature \cite{EHN96}. Noise in the measurement, as described in~(\ref{eq:noisymeasurement}), can lead to unstable reconstructions unless regularization schemes are implemented \cite{EHN96, LP13} which is a viable approach that can be easily adapted to this problem \cite{Bishop95, HS01, LP06}. Among various regularization schemes, Tikhonov regularization stands as a standard technique.
It aims to find $f_\lambda^{\delta}\in X_n$ as a minimizer of the following functional:
  \begin{equation}\label{fig:classicTikhonov-min}
  f_\lambda^{\delta} = \arg\min_{g\in X_n} \|g -f^{\delta}\|_{L^2(\Omega)}^2 + \lambda \|g-f_*\|_{H^{k}(\Omega)} ^2,
  \end{equation}
where $\lambda$ serves as a regularization parameter, $X_n$ represents a generic solution set, and $f_*$ denotes an initial guess. Error bounds of the form $\|f-f_\lambda^{\delta} \|_{H^m(\Omega)}$, with $0\leq m \le k$ have been established in relation to the noise level $\delta$ for various ansatzes of $f_n$ \cite{HS01, BN03, LP06}. In particular, the pioneering work \cite{BN03} considers an approach (\ref{fig:classicTikhonov-min}) when the minimizer $f_n$ is represented as a two-layer neural network and error bound analysis is rigorously carried out. Notably, the penalty term $\|f-f_*\|_{H^{k}(\Omega)} ^2$ in (\ref{fig:classicTikhonov-min}) remains expressed as a Sobolev norm. On the other hand, in our recent joint work \cite{LLMP23}, along with our co-authors,
we propose a novel regularization scheme that incorporates the information of network parameters into the regularization penalty term, which is more generic in computation than that in \eqref{fig:classicTikhonov-min}. This scheme is given by
\begin{align}\label{eq:min_problem_general}
& f_\lambda^{\delta} = \arg\min_{g\in F_n}
  J_{\lambda}(g)
\end{align}
where we again use $f_\lambda^{\delta} $ to denote the minimizer. The above functional $J_{\lambda}(g)$ is defined by
\begin{align}\label{eq:min_problem_exBarron}
  J_{\lambda}(g) := \left\|g -f^{\delta}\right\|_{L^2(\Omega)}^2 +
    \lambda \lr{\frac 1 n \sum_{i=1}^{n}
  \abs{a_{i}}\lr{\norm{w_{i}}{1} + \abs{b_{i}}}^{k}}^{2} ,\qquad g\in  F_{n},
\end{align}
and $F_{n}$ denotes the set of all two-layer neural networks with~$n$ neurons which will be further defined in the subsequent text. The regularization scheme (\ref{eq:min_problem_exBarron}) leverages the neural network structure, incorporating the network parameters into the regularization penalty term.
%{\color{red}\sout{This penalty term is based on the proposed extended Barron norm from \cite{LLMP23} where error bound analysis on $\|f-f_\lambda^{\delta}\|_{H^m(\Omega)}$, with $0\leq m \le k$, is conducted by carefully bounding the Sobolev norm via the {\it a priori} information of the extended Barron norm.}}

In this manuscript, we aim to significantly broaden the scope of the research and introduce a more comprehensive regularization scheme. Specifically, we consider a regularization functional that is formulated as follows:
\begin{equation}\label{eq:min_problem}
  J_{\lambda}(g) := \left\|g -f^{\delta}\right\|_{L^2(\Omega)}^2 +
    \lambda \mathcal{P}(g).
    %,\qquad g\in  F_{n}.
\end{equation}
Here $\mathcal{P}(g)$ represents the penalty term, which adopts various forms depending on the neural network architecture. Furthermore, we delve into the impact of the dimensionality index $d$, particularly focusing on its influence on the approximation performance.
Our objective is to elucidate the inherent relationship between different network norms and the dependency of the dimensionality index on various error bounds. Our analysis reveals that although the universal approximation property could be dimension-independent for function approximation in certain spaces, the dimensionality index inevitably influences the derivative approximation. The manuscript is structured as follows. In Section \ref{se_normoverview}, we overview three distinct network norms: the (extended) Barron norm, the variation norm, and the Radon-BV seminorm. Section \ref{se_differentspace} explores the connections among these norms. Subsequently, in Sections \ref{se_revisitUA} and \ref{se_error}, we revisit the convergence analysis for function approximation and derive error bounds of the function and derivatives approximation for the regularization scheme (\ref{eq:min_problem}), with a particular emphasis on the dependency of the dimensionality of the target function's domain. Finally, a conclusion section ends the paper.

\section{Overview of the shallow neural network norms}\label{se_normoverview}

In this section, we provide a brief overview of three most popular (semi)norms for shallow neural networks (\ref{eq:shallow}), namely the (extended) Barron norm, the variation norm, and the Radon-BV seminorm. These norms can be incorporated as penalty terms during neural network training, which we will explore in Section \ref{se_revisitUA} when simultaneously approximating a function and its derivatives.
To streamline our subsequent analysis, we recall the unit cubic domain $\Omega = (0,1)^d$ with a finite dimensionality index $d$. Without loss of generality, we assume that $d$ is an integer with $d\geq 2$. It is worth noting that, while our current work focuses on a unit cubic domain $\Omega$, extending the analysis to the entire space is a non-trivial task that will be delved into more deeply in future research.

\subsection{(Extended) Barron norm}\label{subse_Barronnormoverview}
The first norm originates from the properties of neural network coefficients and has been well discussed in \cite{EMW22} for the $\relu$ function. Extension to the RePU function can be found in \cite{LLMP23}.

If the activation function is chosen as the ReLU function $\sigma_1$,
a straightforward value that we can use to evaluate the shallow neural network  (\ref{eq:shallow}) is the weighted value of its coefficients. For instance, we can define this value as follows:
\begin{align}\label{eq:exBarronfinite}
\mathcal{B}_p(f) := \left(\frac{1}{n} \sum_{i=1}^n|a_i|^p (\|w_i\| + |b_i|)^p + |a_0|^p \right)^{1/p}.
\end{align}
When the number of the neurons increases as $n\rightarrow \infty$,
by assuming the bias $a_0=0$, one can focus on a continuous version of (\ref{eq:shallow}) where the function $f(x)$ admits
\begin{equation}\label{eq:representation_Barron_spaces}
    f(x) = \int_{\R\times \R^d \times \R}a\sigma_1(w\cdot x + b)\rho(\di a, \di w, \di b),\quad x\in \Omega,
\end{equation}
and the $\rho$ represents a probability distribution for the network coefficients $(a,w,b)\in \R\times \R^d \times \R$.
Referring to \cite{EMW22}, we can express the Barron norm in the following form
$$
\|f\|_{B_{p}}=\inf_{\rho}\left( \EE_{\rho} \left[ |a|^{p}\left(\|w\|_1+|b|\right)^{p} \right] \right)^{\frac{1}{p}},\quad 1\le p< \infty,
$$
where the infimum is taken over all $\rho$ satisfying \eqref{eq:representation_Barron_spaces}.
Such a formulation can be viewed as a continuous analogy of (\ref{eq:exBarronfinite}).
When $p=\infty$, the Barron norm is defined by
$$
\|f\|_{B_{\infty}} = \inf_{\rho} \max_{(a,w,b)\in \supp(\rho)} |a|\left(\|w\|_1+|b|\right).
$$
Following the arguments in \cite{EMW22}, one can define the Barron space $B_{p}$ as the set of continuous functions that can be represented by \eqref{eq:representation_Barron_spaces} with finite Barron norm.
It has been proven in \cite{EMW22} that, for any $1\le p\le \infty$, $B_p=B_{\infty}$ and $\|f\|_{B_{p}}=\|f\|_{B_{\infty}}$.
In particular, \cite{EMW22} also demonstrates that the Barron space is closely related to Reproducing Kernel Hilbert Spaces (RKHS).
%{\color{red}\sout{such that$B_2 = \bigcup_{\pi\in P(S^{d})} \Hs_{k_{\pi}}.$Here we denote $S^{d}:=\left\{ \hat{w}:=(w,b)\in \R^{d+1}: \|\hat{w}\|_1=1 \right\}$ and $\Sigma_{S^{d}}$ is the Borel $\sigma$-algebra on $S^d$. The set $P(S^d)$ is probability measures over $(S^{d}, \Sigma_{S^{d}})$ and $k_{\pi}$ is the kernel defined as $k_{\pi}(x,x')=\EE_{\hat{w}\sim \pi}\left[\sigma(\hat{w}\cdot (x,1))\sigma(w\cdot(x',1))\right]$.The RKHS induced by the kernel function $k_{\pi}$ is denoted as $\Hb_{k_{\pi}}$.}}

When the activation function index is an integer $k\geq 1$, one can extend the Barron norm. The following definition is proposed in \cite{LLMP23} to accommodate this extension:

\begin{definition}[Extended Barron spaces \cite{LLMP23}]\label{def:extendedBarron}
    Let a function $f:\Omega\to \R$ admit
    \begin{equation}\label{eq:representation_extended_Barron_spaces}
        f(x) = \int_{\R\times \R^{d}\times \R} a\sigma_k(w\cdot x+b)\rho(\di a,\di w, \di b),\quad x\in\Omega,
    \end{equation}
    where the activation function is $\sigma_k(z) = \left(\max\left\{0,z\right\}\right)^k$ and $\rho$ is a probability distribution.
    Then, for a multi-index $\alpha$ with $|\alpha|:=\sum_{i=1}^{d}|\alpha_i|\le k$, we derive
    $$
    \partial^{\alpha} f(x)=\int_{\R\times \R^{d}\times \R}a\sigma_{k-|\alpha|}(w\cdot x+ b)\rho_{\alpha}(\di a,\di w,\di b),
    $$
    where $\rho_{\alpha}$ is the pushforward measure induced by the continuous map $(a,w,b)\mapsto \left(\frac{k!}{(k-|\alpha|)!}aw^{\alpha}, w, b\right)$ with $w^{\alpha}=\prod_{i=1}^{d}w_i^{\alpha_i}$.
    
    Define $
    \|f\|_{B_{p,\rho}^{k}}:=\left(\EE_{\rho}\left[|a|^{p}\left(\|w\|_1+|b|\right)^{kp}\right]\right)^{\frac{1}{p}},
    $
    for $1\le p<\infty$ and, for $p=\infty$
    $$
    \|f\|_{B_{\infty,\rho}^{k}}:=\max_{(a,w,b)\in \supp(\rho)}|a|\left(\|w\|_1+|b|\right)^{k}.
    $$
    The extended Barron norm is defined by
    $
    \|f\|_{B_{p}^{k,m}}:=\inf_{\rho}\left(\sum_{|\alpha|\le m}\|\partial^{\alpha}f\|_{B_{p,\rho_{\alpha}}^{k-|\alpha|}}^{p}\right)^{\frac{1}{p}}
    $
    for $1\le p<\infty$ and, 
    $
    \|f\|_{B_{\infty}^{k,m}}:=\inf_{\rho}\max_{|a|\le m}\|\partial^{\alpha}f\|_{B_{\infty,\rho_{\alpha}}^{k-|\alpha|}}$ for $p=\infty$, 
    where the infimum is taken over all $\rho$ satisfying \eqref{eq:representation_extended_Barron_spaces} and $m\le k$.
    Extended Barron spaces $B_{p}^{k,m}$ are the set of functions that can be represented by \eqref{eq:representation_extended_Barron_spaces} with finite extended Barron norm.
\end{definition}

It is proven in \cite{LLMP23} that $B_1^{k,0}$ is a normed space and for all $0\leq m \leq k$ and $1\leq p \leq \infty$ there holds $B_p^{k,m} = B_1^{k,0}$ as sets.
For brevity, in the following contents, we will use $B_1^{k}$ to denote $B_{1}^{k,0}$. Furthermore, the (extended) Barron norm can be defined on the entire space or restricted to any bounded domain. Specifically, we employ the notation $B_1^{k}(\Omega)$ to accommodate our specific setting, where $\Omega = (0,1)^d$.

\subsection{Variation norm}\label{subse_Variationnormoverview}

The second norm originates from the convex hull of the dictionary sets. We cite the following definition and refer directly to \cite{SX23} for further details.
\begin{definition}[Variation spaces of $\repu$ dictionary \cite{SX23}]\label{def:variation_space}
    For a bounded domain $\Omega\subset \R^d$, denote the dictionary as
    $$
    \PP_k=\left\{ \sigma_k(w\cdot x+b): w\in \Sf^{d-1},b\in[c_1,c_2] \right\}\subset L^{2}(\Omega),
    $$
    where $c_1 < \inf\left\{w\cdot x: x\in\Omega, w\in \Sf^{d-1}\right\}<\sup \left\{w\cdot x : x\in\Omega, w\in\Sf^{d-1}\right\}<c_2$
    and $\Sf^{d-1}:=\left\{w\in \R^{d}\mid \|w\|_2=1\right\}$.
    The variation norm of $f$ is defined by
    $$
    \|f\|_{\PP_k}:=\inf\left\{c>0: \frac{f}{c}\in\overline{\conv(\pm \PP_k)}\right\},
    $$
    where
    $$
    \overline{\conv(\pm \PP_k)} = \overline{\left\{\sum_{j=1}^{n}a_jh_j: n\in\N, h_j\in \PP_k, \sum_{i=1}^{n}|a_i|\le 1\right\}}.
    $$
    The variation space $\K(\PP_k)$ is
    $
    \K(\PP_k):=\left\{ f\in L^2(\Omega):\|f\|_{\PP_k}<\infty \right\}.
    $
\end{definition}

As it has been discussed in \cite{SX23} if the dictionary $\PP_k$ is replaced by
$$\FF_s=\left\{(1+|\omega|)^{-s}\e^{2\pi i \omega\cdot x}: \omega\in\R^{d} \right\}\subset L^2(\Omega),$$
the variation norm is equivalent to the Barron spectral norm, as initially proposed in \cite{Barron93} and subsequently explored in \cite{Xu20} and \cite{SX23}.

\subsection{{Radon-BV seminorm}}\label{subse_RadonBVoverview}
The third one focuses primarily on the coefficients of neural networks. In contrast to the (extended) Barron norm, this seminorm disregards all bias coefficients and solely considers the non-bias weights of \eqref{eq:shallow}, for instance, adopting a formulation such as
 $\frac{1}{2}\sum_{i=1}^{n}\left( a_i^2 + \|w_i\|_2^2 \right) $
as discussed in \cite{SESS19}.
In an alternative form, \cite{SESS19} proposes controlling the output weight under the constraint of hidden weights, such that the definition of $\mathcal{V}(\theta)$, with the coefficient set $\theta=(n,a,w,b)$,
is given by
\begin{align}\label{eq:RadonBVfinite}
\mathcal{V}(\theta):= \inf_{\theta\in\Theta; h_\theta=f ; \|w_i\|_2=1} \sum_{i=1}^{n}|a_i|,
\end{align}
where $\Theta=\left\{\theta=(n,a,w,b)\mid n\in \N, a\in \R^{n+1}, w=(w_i)_{i=1}^{n}\in \R^{d\times n} \text{ and } b\in \R^{n}\right\}$
and $h_{\theta}(x)=\sum_{i=1}^{n}a_i\sigma(w_i\cdot x + b_i)+a_0$.

Similar to the previous Subsection \ref{subse_Barronnormoverview}, when the number of neurons tends to infinity as $n\rightarrow \infty$,
\cite{SESS19} introduces a signed measure $\mu$ on $\Sf^{d-1}\times \R$ and defines an auxiliary function
$
h_{\mu,a_0}:=\int_{\Sf^{d-1}\times \R} \sigma_k (w\cdot x + b) \di \mu(w,b) + a_0,
$
with a bias $a_0\in\R$. Subsequently, the finite summation (\ref{eq:RadonBVfinite}) can be extended to
an infinite form $\bar{\mathcal{V}}(f)$ as
$$
\bar{\mathcal{V}}(f) := \inf_{\mu,a_0;h_{\mu,a_0} = f} \|\mu\|_{\M(\Sf^{d-1}\times \R)},
$$
where $\|\cdot\|_{\M(\Sf^{d-1}\times \R)}$ is the total variation in the measure sense. Observing that $\bar{\mathcal{V}}(f)$ eliminates the influence of the bias term from the integral, the discussion in \cite{OSWS20} further removes an affine term and defines the following:
$$\bar{\mathcal{V}}_1(f) := \min_{f=h_{\mu,v,a_0}; \mu\in\M(\Sf^{d-1}\times \R);  v\in\R^d; a_0\in\R} \|\mu\|_{\M(\Sf^{d-1}\times\R)},$$
where
$
h_{\mu,v,a_0} :=\int_{\Sf^{d-1}\times \R} \left( \sigma_k(w\cdot x - b)-\sigma_k(-b) \right)\di\mu(w,b)+v\cdot x +a_0.
$
If the activation function is fixed as the ReLU, \cite{OSWS20} highlights significance of the Radon transform and introduces a specific $\Rnorm$-norm for any Lipschitz continuous function $f:\R^d\to \R$ by
$$
\|f\|_{\Rnorm}:=\sup\left\{-\gamma_d\left< f,(-\Delta)^{\frac{d+1}{2}}\Radon^*\varphi \right>\mid \varphi\in\Sw(\Sf^{d-1}\times \R),\varphi \text{ is even},\|\varphi\|_{\infty}\le 1\right\}
$$
where $\gamma_{d} = \frac{1}{2(2\pi)^{d-1}}$. Here
$\Radon^{*}$ is the dual of the Radon transform defined by
$$
\Radon^{*} \varphi (x) = \int_{\Sf^{d-1}} \varphi(\alpha, \alpha \cdot x) \di \alpha,
$$
and
$\Ft((-\Delta)^{\frac{s}{2}}g)(\eta) = |\eta|^s\Ft(g)(\eta)$, where $\Ft$ is the Fourier transform.
If $f$ is not Lipschitz continuous, one can set $\|f\|_{\Rnorm}=+\infty$.
As demonstrated in \cite[Theorem 1]{OSWS20}, $\bar{\mathcal{V}}_1(f) = \|f\|_{\Rnorm}$ for all functions $f$.
Additional properties of the $\Rnorm$-norm can be found in \cite{OSWS20}, and we omit the details here.

If the activation function is chosen as $\repu$, one can further eliminate the influence of low-order polynomials, and \cite{PN21} defines a Radon-BV space by introducing
\begin{definition}[Radon-BV spaces on $\R^d$ \cite{PN21}]\label{def:RadonBV}
    Define $\RBV_{k+1}:=c_d\partial_{t}^{k+1}\Lambda^{d-1}\Radon$,
    with $c_d:=\frac{1}{2(2\pi)^{d-1}}$.
    $\Lambda^{d}$ is a ramp filter such that, for any given function $\Phi(\alpha, t)$, it is defined by
    $$
    \Lambda^{d}\left\{\Phi\right\}(\alpha,t):=
    \begin{cases}
        \partial_{t}^{d}\Phi(\alpha, t),& d\text{ is even},\\
        \Hb_t\partial_{t}^{d}\Phi(\alpha, t),& d\text{ is odd}.
    \end{cases}
    $$
    Here $\Hb_t$ is the Hilbert transform in the variable $t$, $\partial_t$ is the partial derivative with respect to $t$,
    and $\Radon$ is the Radon transform.
    The Radon-BV space is defined by
    $$
    \F_{k+1}:=\left\{f\in L^{\infty, k}(\R^{d}) : \RBV_{k+1} f\in \M(\Sf^{d-1}\times \R)\right\},
    $$
    where $L^{\infty,k}(\R^d)$ consists all functions $f$ with finite $\esssup_{x\in\R^d}|f(x)|(1+\|x\|_2)^{-k}$, and $\M(\Sf^{d-1}\times \R)$ is the Banach space of finite Radon measures on $\Sf^{d-1}\times \R$.
\end{definition}

In principle, there is no explicit requirement on the representation of $f\in\F_{k+1}$ as it is implicitly determined by the operator $\RBV_{k+1}$.
\cite{PN21} demonstrates  that $\RBV_{k+1}$ can sparsify the $\repu$ neurons as
$$
\RBV_{k+1}\left(\frac{\sigma_{k}\left(w \cdot -b\right)}{k!}\right) = \frac{\delta_{\Sf^{d-1}\times \R}\left(\cdot - z\right)+(-1)^{k+1}\delta_{\Sf^{d-1}\times \R}\left(\cdot + z\right)}{2},
$$
where $w\in\Sf^{d-1}$, $b\in\R$, $z=(w,b)$,
and $\delta_{\Sf^{d-1}\times \R}$ is the Dirac impulse.
Furthermore, the null space of $\RBV_{k+1}$ in $\F_{k+1}$ comprises all polynomials with a degree strictly less than $k+1$.
%{\color{red}\sout{Denote the dimension of this null space by $N_0$ and let $(\phi_i,p_i)_{i=1}^{N_0}$ represents a biorthogonal system for the null space. Therefore, any function $f\in\F_{k+1}$ can be decomposed into two components. The first component is an integral of $\repu$ neurons, while the second component is a low-order polynomial, specifically $f=\RBV_{k+1,\phi}^{-1}\RBV_{k+1}f + \sum_{i=1}^{N_0}\left<\phi_i,f\right>p_i,$ where $\RBV_{k+1,\phi} {-1}\varphi=\int_{\Sf^{d-1}\times \R} \left( \frac{\sigma_{k}(w\cdot -b)}{k!}-\sum_{j=1}^{N_0}\left<\phi_j,\frac{\sigma_{k}(w\cdot -b)}{k!}\right>p_j\right)\di \varphi(w,b).$  More details can be found in   \cite[Theorem 22]{PN21}. }}

The following proposition illustrates how the operator $\RBV_{k+1}$ transforms an infinitely wide neural network into a measure defined on the Radon domain.
\begin{proposition}\label{ppt:RBV}
    Let $f$ be a function satisfying the following representation
    \begin{equation}\label{eq:representation_RadonBV}
        f(x) = \int_{\Sf^{d-1}\times \R} \sigma_{k}(w\cdot x-b)\di\mu(w,b),\quad \mu\in \M(\Sf^{d-1}\times \R).
    \end{equation}
    Then
    $$
    \frac{1}{k!}\RBV_{k+1}f = \begin{cases}
        \mu^{+}, & k\text{ is odd},\\
        \mu^{-}, & k\text{ is even},
    \end{cases}
    $$
    where $\mu^{+}(w,b) :=  \frac{1}{2} \left(\mu(w,b) + \mu(-w,-b)\right)$
    and $\mu^{-}(w,b) :=  \frac{1}{2} \left(\mu(w,b) - \mu(-w,-b)\right)$.
\end{proposition}
\begin{proof}
    Given any test function $\varphi\in\Sw(\Sf^{d-1}\times \R)$, the following equality holds
    \begin{equation*}
        \begin{split}
            \left<\frac{1}{k!}\RBV_{k+1}f,\varphi\right> &= \frac{1}{k!}\left< f, \RBV_{k+1}^{*}\varphi \right> =\frac{1}{k!}\left< \int_{\Sf^{d-1}\times \R} \sigma_{k}\left(w\cdot -b\right)\di\mu(w,b), \RBV_{k+1}^{*}\varphi \right>\\
%            &=\int_{\Sf^{d-1}\times \R} \frac{1}{k!}\left<\sigma_{k}(w\cdot -b), \RBV_{k+1}^{*}\varphi\right> \di\mu(w,b)\\
            &=\int_{\Sf^{d-1}\times \R} \left<\frac{1}{k!}\mathcal{R}_{k+1}\sigma_{k}(w\cdot -b), \varphi\right> \di\mu(w,b)\\
            &=\left<\int_{\Sf^{d-1}\times \R} \frac{1}{2}\left(\delta_{\Sf^{d-1}\times \R}(\cdot - z)+(-1)^{k+1}\delta_{\Sf^{d-1}\times\R}(\cdot + z) \right) \di\mu(w,b),
            \varphi\right>
        \end{split}
    \end{equation*}
    which yields
    $$
    \int_{\Sf^{d-1}\times \R} \left( \frac{1}{2}\delta_{\Sf^{d-1}\times \R}(\cdot - z)+(-1)^{k+1}\frac{1}{2}\delta_{\Sf^{d-1}\times\R}(\cdot + z) \right) \di\mu(w,b)
    =\begin{cases}
        \mu^{+},& k\text{ is odd},\\
        \mu^{-},& k\text{ is even},
    \end{cases}
    $$
    and ends the proof.
\end{proof}

We observe that the Radon-BV space, as defined in Definition \ref{def:RadonBV} is established on the entire domain $\R^d$. However, for any bounded domain $\Omega$, specifically $\Omega=(0,1)^d$, \cite{PN23} offers an expanded or modified definition.
\begin{definition}[Radon-BV spaces on $\Omega$ \cite{PN23}]
    The Radon-BV space on a bounded domain $\Omega\subset \R^{d}$ is defined by
    $$
    \F_{k+1}(\Omega):=\left\{ f\in \D'(\Omega): \exists~ g\in\F_{k+1} \text{ s.t. } g|_{\Omega}=f \right\},
    $$
    where $\D'(\Omega)$ is the distributions on $\Omega$.
\end{definition}
The notation of $\D'(\Omega)$ can be found, for instance, in \cite[1.57]{Adams_book}.
Based on this definition, one can define a seminorm, called Radon-BV seminorm on $\F_{k+1}(\Omega)$, that
$$
|f|_{\F_{k+1}(\Omega)} := \inf_{g\in \F_{k+1}}\left\{\left\|\frac{1}{k!}\RBV_{k+1}g\right\|_{\M(\Sf^{d-1}\times \R)}: g|_{\Omega}=f\right\},
$$
for $f\in\F_{k+1}(\Omega)$.
It is easy to check that $\|\cdot\|_{L^2(\Omega)} + |\cdot|_{\F_{k+1}(\Omega)}$ is a norm for $\F_{k+1}(\Omega)$.
If the activation function is chosen as ReLU with $k=1$, \cite[Lemma 2]{PN23} demonstrates that,
for any $f\in \F_{2}(\Omega)$,
there exists an extension $f_{\ext}\in\F_{2}$ that satisfies the given conditions
$$
f_{\ext}(x) = \int_{\Sf^{d-1}\times \R} \sigma_1(w\cdot x -b)\di\mu(w,b) + P(x),\quad |f|_{\F_{2}(\Omega)} = \|\RBV_{2}f_{\ext}\|_{\M(\Sf^{d-1}\times \R)},
$$
where $\supp \mu\subset \Sf^{d-1}\times [-\diam(\Omega), \diam(\Omega)]$ and $P(x)$ is an affine function, with $\diam(\Omega)=\sup_{x\in \Omega}\|x\|_2$.
% \it{i.e.} $P\in \Poly_1$.
This result can be generalized to $f,f_{\ext}\in \F_{k+1}$ for any integer $k\geq 1$ and the RePU activation function $\sigma_k$. Specifically, by denoting
$P(x)$ as a polynomial with a degree strictly less than $k+1$, we have the following property:
\begin{equation}\label{eq:extension}
    f_{\ext}(x) = \int_{\Sf^{d-1}\times \R} \sigma_k(w\cdot x -b)\di\mu(w,b) + P(x),
\end{equation}
with
$$
|f|_{\F_{k+1}(\Omega)} = \left\|\frac{1}{k!}\RBV_{k+1} f_{\ext}\right\|_{\M(\Sf^{d-1}\times \R)},
$$
where $\supp \mu\subset \Sf^{d-1}\times [-\diam(\Omega), \diam(\Omega)]$.
Furthermore, if $k$ is odd, then $\mu$ is even; whereas if $k$ is even, then $\mu$ is odd.

\section{Relationships among different norms and spaces}\label{se_differentspace}
In this section, we will systematically explore the relationship among various network norms overviewed earlier. To approximate the derivatives of an unknown function, it is crucial to establish a link between these network norms and the standard Sobolev norm. Particularly, we will monitor the impact of the dimensionality index $d$ in various relationships. This dimensionality index will be explicitly emphasized in the error estimates when both the function and its derivatives are approximated. To accomplish this goal, it is crucial to investigate the intricate relationships that exist among various norms and spaces, including Sobolev norms and spaces.

\subsection{Relationships among different network norms}
We first delve into the relationships among different network norms. In doing so, we will draw upon and extend some existing results to aid in our discussion.

Regarding the relationship between the variation norm and the (extended) Barron norm, we extend the findings in \cite{SX23}, which solely considered the case of $k=1$ while neglecting the dimensionality parameter $d$. Given that the dimensionality index $d$ is at least $2$, we formulate the following proposition:
\begin{proposition}\label{prop:Barron_variation}
    For a unit cubic domain $\Omega=(0,1)^{d}$, the following holds
    \begin{equation}\label{eq:Baaron_var}
        C\left\|f\right\|_{\mathbb{P}_{k}\left(\Omega\right)}\le \left\|f\right\|_{B_{1}^{k}(\Omega)}\le C'd^{\frac{k}{2}}\left\|f\right\|_{\mathbb{P}_{k}},
    \end{equation}
    with $C,C'>0$ solely depending on $k$. 
    % \begin{equation}\label{eq:Barron_var}
    %     \|f\|_{B_{1}^{k}(\Omega)}\sim d^{\frac{k}{2}}\|f\|_{\PP_{k}},\quad d\to\infty.
    % \end{equation}
    Here, the parameters $-c_1=c_2=\diam(\Omega)=\sqrt{d}$ are chosen for the dictionary $\PP_k$.
    % The above notation $\sim$ means that the two values are in the same order. %\it{i.e.} $a\sim b$ if there are positive constants $c_1$ and $c_2$ such that $c_1 b\le a\le c_2 b$.
    % The implied constant in the $\sim$ notation in \eqref{eq:Barron_var} solely depends on $k$.
\end{proposition}
\begin{proof}
    Denote the dictionary
    $
    \B=\left\{ (\|w\|_1 + |b|)^{-k}\sigma_k(w\cdot x + b): w\in\R^{d}, b\in\R \right\}.
    $
    We define $\|f\|_{\B}:=\inf\left\{c>0: f/c\in\overline{\conv(\pm\B)}\right\} = \|f\|_{B_1^k(\Omega)}$.

    On the one hand, the following holds
    $$
    (\|w\|_1+|b|)^k\le \left(\sqrt{d} + \max(|c_1|, |c_2|) \right)^k \le 2^k d^{\frac{k}{2}},
    $$
    for $\Omega = (0,1)^d$, $w\in \Sf^{d-1}$ and $b\in[c_1,c_2]$. Thus we have proven
 $\PP_k\subset 2^k d^{\frac{k}{2}}\overline{\conv(\pm\B)} $.

 On the other hand, if $w=\vec{0}$, the element in $\B$ is either a constant $0$ (when $b\leq 0$) or $1$ ($b>0$).
    Since the constant $0$ undoubtedly belongs to $\overline{\conv(\pm\PP_{k})}$, our focus shifts to the constant $1$.
    Let us consider a special univariate polynomial
    $$(z+b)^{k} = \sum_{j=0}^{k}C_{k}^{j}b^{j} z^{k-j},\quad C_{k}^{j}=\frac{k!}{(k-j)!j!}.$$
    If we select $\{b_i\}_{i=1}^{k+1}$ such that $c_1<b_1<b_2<\cdots<b_{k+1}<c_2$,
    we can construct a matrix
    $$
    A :=
    \begin{bmatrix}
    C_{k}^{0}   & C_{k}^{1} b_1 & \cdots    & C_{k}^{k} b_{1}^{k} \\
    C_{k}^{0}   & C_{k}^{1} b_2 & \cdots    & C_{k}^{k} b_{2}^{k} \\
    \vdots      & \vdots        & \ddots    & \vdots \\
    C_{k}^{0}   & C_{k}^{1} b_{k+1} & \cdots & C_{k}^{k} b_{k+1}^{k}
    \end{bmatrix}.
    $$
    The determinant of $A$ is nonzero due to the Vandermonde matrix's determinant form and the multiplicative property of determinants.
    This implies that $\{z+b_i\}_{i=1}^{k+1}$ forms a basis for univariate polynomials of degree strictly less than $k+1$.
    Consequently, the constant $1$ can be expressed as a linear combination of $(w_0\cdot x+b_1)^{k}$, $(w_0\cdot x+b_2)^{k}$, $\cdots$, $(w_0\cdot x + b_{k+1})^k$ for any $w_0\in\Sf^{d-1}$.
    Furthermore, $(w_0\cdot x + b_i)^{k}$ can be expressed as $\sigma_{k}(w_0\cdot x + b_i)-\sigma_{k}(-w_0\cdot x -b_i)$.
    Therefore, we can conclude that $1\in C(k)\overline{\conv(\pm\PP_k)}$ for some constant $C(k)$ solely related to $k$.

    If $w\neq\vec{0}$, let us consider
    $$
    (\|w\|_1+|b|)^{-k}\sigma_k(w\cdot x +b) = \left(\frac{\|w\|_2}{\|w\|_1+|b|}\right)^k \sigma_k\left(\frac{w}{\|w\|_2}\cdot x + \frac{b}{\|w\|_2}\right).
    $$
    If $\frac{b}{\|w\|_2}<c_1$, the above element is $0$.
    If $\frac{b}{\|w\|_2}\in[c_1,c_2]$, the above element is in $\overline{\conv(\pm\PP_k)}$ since $\frac{\|w\|_2}{\|w\|_1+|b|}\le 1$.
    If $\frac{b}{\|w\|_2}>c_2$, the above element can be interpreted as a polynomial on $\Omega$ %and
    % $\left(\|w\|_1+|b|\right)^{-k}\left(w\cdot x + b\right)^{k}= \left(\frac{w}{\|w\|_1+|b|} \cdot x + \frac{b}{\|w\|_1+|b|} \right)^k\le (\diam(\Omega) + 1)^k = (\sqrt{d} + 1)^{k}$. Finally $(\sqrt{d} + 1)^{k}\in C_3(k)d^{\frac{k}{2}}\overline{\conv(\pm\PP_{k})}$ for some constant $C_3(k)$ depending only on $k$ because $1\in C(k)\overline{\conv\left(\pm \PP_k\right)}$.
    as 
    \begin{equation*}
        \begin{split}
            \left(\left\|w\right\|_1+|b|\right)^{-k}\left(w\cdot x + b\right)^{k} & = \left(\frac{\left\|w\right\|_2}{\left\|w\right\|_1+|b|}\right)^{k}\left(\frac{w}{\left\|w\right\|_2}\cdot x + \frac{b}{\left\|w\right\|_2}\right)^{k}\\
            & \le 2^{k}\left(\frac{\left\|w\right\|_2}{\left\|w\right\|_1+|b|}\right)^{k}\left(\frac{w}{\left\|w\right\|_2}\cdot x\right)^{k}+2^{k}\left(\frac{b}{\left\|w\right\|_1+|b|}\right)^{k}.
        \end{split}
    \end{equation*}
    The first term in the above equation holds
    \[
    2^{k}\left(\frac{\left\|w\right\|_2}{\left\|w\right\|_1+|b|}\right)^{k}\left(\frac{w}{\left\|w\right\|_2}\cdot x\right)^{k}\in 2^{k+1}\overline{\conv\left(\pm \mathbb{P}_{k}\right)},
    \]
    since $\left|\frac{\left\|w\right\|_2}{\left\|w\right\|_1+|b|}\right|\le 1$ and $\left(\frac{w}{\left\|w\right\|_2}\cdot x\right)^{k}=\sigma_{k}\left(\frac{w}{\left\|w\right\|_2}\cdot x \right)-\sigma_{k}\left(-\frac{w}{\left\|w\right\|_2}\cdot x\right)$.
    The second term holds
    \[
    2^{k}\left(\frac{b}{\left\|w\right\|_1+|b|}\right)^{k}\le 2^{k}\in 2^{k}C(k)\overline{\conv\left(\pm \mathbb{P}_{k}\right)},
    \]
    because $1\in C(k)\overline{\conv\left(\pm \mathbb{P}_{k}\right)}$.

    To conclude, we have established that
    \[
    \PP_k\subset C(k)d^{\frac{k}{2}}\overline{\conv(\pm\B)} \textrm{ and }
    \B\subset C(k)\overline{\conv(\pm\PP_k)}, 
    \]
    % $$\PP_k\subset C(k)d^{\frac{k}{2}}\overline{\conv(\pm\B)} \textrm{ and }
    % \B\subset C(k)d^{\frac{k}{2}}\overline{\conv(\pm\PP_k)}, $$
    for a generic constant $C(k)$ solely depending on $k$,
    thereby obtaining the desired result in Proposition \ref{prop:Barron_variation}.
\end{proof}

Regarding the connection between the Radon-BV seminorm and the variation norm, \cite{SX23} has established the following relationship.
More precisely, for a given $k$ and any function $f$ that belongs to both the Hilbert space $L^2(\Omega)$ and the Radon-BV space $\F_{k+1}(\Omega)$ while also residing in the variation space $\mathcal{K}(\PP_k)$, the following inequalities hold
\begin{align}
   & \|f\|_{L^2(\Omega)}+|f|_{\F_{k+1}(\Omega)} \le \left(k!+\sup_{g\in\PP_{k}}\|g\|_{L^2(\Omega)}\right) \|f\|_{\PP_{k}}, \nonumber \\
  & \|f\|_{\PP_{k}(\Omega)} \le \frac{1}{k!}\left( 1+M(d) \sup_{g\in\PP_k}\|g\|_{L^2(\Omega)} \right)\|f\|_{L^2(\Omega)}+M(d)|f|_{\F_{k+1}(\Omega)}, \label{eq:var_RBV}
\end{align}
where $M(d)$ is the smallest value which satisfies
$\|p\|_{\PP_{k}} \le M(d)\|p\|_{L^2(\Omega)}$ for all $p\in \Poly_k$. Here $\Poly_k$ is the set of all polynomials with a degree strictly less than $k+1$. Such a constant $M(d)$ implicitly depends on the dimensionality parameter $d$.
The existence of $M(d)$ is ensured by the finite dimension of $\Poly_k$.
For the specific case of $\Omega=(0,1)^d$ we will provide a lower bound for $M(d)$ in the subsequent Lemma \ref{lem:M_lowerbound}.
Meanwhile the $\sup_{g\in \PP_{k}}\|g\|_{L^2(\Omega)}$ can be bounded by $2^{k}d^{\frac{k}{2}}$ since, for any $g\in \PP_k$,
\begin{equation}\label{eq:norm_g}
        \begin{split}
            \|g\|_{L^{2}(\Omega)}^2
              =\int_{\Omega} \left(\sigma_{k}\left(w\cdot x + b\right)\right)^2 \di x \le \int_{\Omega} \left(\|w\|_2^2+1\right)^{k}\left(\|x\|_2^2+b^2\right)^{k} \di x \le 2^{2k}d^{k}.
        \end{split}
    \end{equation}

\subsection{Relationship to the Sobolev norm}
%{\color{red}\sout{It is noteworthy that the relationships among various network norms are influenced by the dimensionality index $d$, whether implicitly or explicitly. Additionally,}}
In this subsection, we aim to establish a connection between these network norms and the standard Sobolev norm. Such a relationship is crucial for conducting error estimates of the regularization scheme, particularly when approximating a function and its derivatives.

In \cite{LLMP23}, we, along with our co-authors, demonstrated that extended Barron spaces can be embedded into Sobolev spaces. Here, we refine this result by specifically considering $\Omega=(0,1)^d$ and explicitly evaluating the impact of the dimensionality index $d$ on the embedding's performance.

\begin{lemma}\label{lem_barronsobolev}
    Given any fixed integer $k$ and a fixed domain $\Omega=(0,1)^d$, there exists a constant $C(d,m{\color{red},} k)$ such that for all $f\in B_{1}^{k}(\Omega)$,
    the following holds
    $$
    \|f\|_{H^{m}(\Omega)}\le C(d,m,k)\|f\|_{B_{1}^{k}(\Omega)},\quad 0\le m\le k
    $$
    and $C(d,m,k)\sim d^{\frac{m}{2}}$ asymptotically as $d\to \infty$.
\end{lemma}
\begin{proof}
    The generic inequality has been shown in \cite[Thm 7]{LLMP23}.
    Now, we focus on $f\in B_{1}^{k}(\Omega)$ with $\Omega=(0,1)^d$. By the arguments in Subsection \ref{subse_Barronnormoverview},
    there is a probability $\tilde{\rho}$ on $\{-1,1\}\times \Sf_1^{d}$ such that
    $$
    f(x) = c_{k,\rho}(f)\int_{\{-1,1\}\times S^{d}} \tilde{a}\sigma_{k}\left(\tilde{w}\cdot x + \tilde{b}\right) \di\tilde{\rho}\left(\tilde{a},\tilde{w},\tilde{b}\right),
    $$
    where $S^{d}:=\left\{(w,b)\in \R^d\times \R: \|w\|_1+|b|=1\right\}$ and $c_{k,\rho}(f)\le (1+\epsilon)\|f\|_{B_1^{k}(\Omega)}$ for an arbitrary small $\epsilon>0$.

    Let $\alpha$ be a multi-index.
    Define $|\alpha|=\sum_{i=1}^{d}\alpha_{i}$,
    $w^{\alpha} = \prod_{i}^{d}w_i^{\alpha_i}$ and
    $\partial^{\alpha}f = \partial_{x_1}^{\alpha_1}\cdots\partial_{x_d}^{\alpha_d}f$ by
    \begin{equation*}
        \begin{split}
            \partial^{\alpha} f(x) = c_{k,\rho}(f)\int_{\{-1,1\}\times S^{d}} \frac{k!}{\left(k-|\alpha|\right)!} \tilde{a}\tilde{w}^{\alpha} \sigma_{k-|\alpha|}\left(\tilde{w}\cdot x + \tilde{b}\right) \di \tilde{\rho}\left(\tilde{a},\tilde{w},\tilde{b}\right).
        \end{split}
    \end{equation*}
    Noticing $\|x\|_{\infty}\le 1$ and $|\tilde{w}^{\alpha}|\le 1$, we derive, for $\Omega=(0,1)^d$,
    \begin{equation*}
        \begin{split}
            \int_{\Omega}\left( \int_{\{-1,1\}\times \Sf_1^{d}} \tilde{a}\tilde{w}^{\alpha}\sigma_{k-|\alpha|}\left(\tilde{w}\cdot x + \tilde{b}\right) \di\tilde{\rho}\left(\tilde{a},\tilde{w},\tilde{b}\right) \right)^2 \di x \le 1.
        \end{split}
    \end{equation*}
    Therefore
    \[
    \|\partial^{\alpha} f\|_{L^2(\Omega)}^2 \le \left(c_{k,\rho}(f)\right)^2\left(\frac{k!}{(k-|\alpha|)!}\right)^2
    \]
    and
    \[
    \|f\|_{H^{m}(\Omega)} \le c_{k,\rho}(f) \sqrt{ \sum_{|\alpha|\le m}\left(\frac{k!}{(k-|\alpha|)!}\right)^2 }.
    \]
    Let
    \begin{equation*}
        \begin{split}
            C(d,m,k):= & \sqrt{ \sum_{|\alpha|\le m}\left(\frac{k!}{(k-|\alpha|)!}\right)^2 }
             = \sqrt{ \sum_{s=0}^{m} \frac{(s+d-1)!}{(d-1)!s!} \left(\frac{k!}{(k-s)!}\right)^2 }
             =:  \sqrt{ \sum_{s=0}^{m} C_s}.
        \end{split}
    \end{equation*}
    We observe that there holds, recursively,
    $$
    \frac{k-1+d}{k}\le \frac{C_{s+1}}{C_{s}} = \frac{(s+d)(k-s)^2}{s+1} \le dk^2,\quad 0\le s < k,
    $$
    and $C_0=1$. Therefore, we conclude the asymptotical behaviour of $C\left(d,m,k\right)\sim d^{\frac{m}{2}}$ as $d\to \infty$.
\end{proof}

It can be observed that when $m=0$ the above Lemma \ref{lem_barronsobolev} yields
\begin{align}\label{eq:L2_B1k}
    \|f\|_{L^2(\Omega)}\le C(k) \|f\|_{B_{1}^{k}(\Omega)}
\end{align}
where the dimensionality index $d$ does not exert any influence on the embedding inequality. However, this behavior does not apply to the other two network norms.

Let $h\in\PP_k$, \textit{i.e.} $h(x)=\sigma_{k}(w\cdot x + b)$, and denote a multi-index $\alpha$ satisfying $|\alpha|\le k$. We thus obtain
$$
\partial^{\alpha}h(x) = \frac{k!}{\left(k-|\alpha|\right)!}w^{\alpha}\sigma_{k-|\alpha|}\left(w\cdot x+b\right),
$$
where $\sigma_{0}$ is the Heaviside step function,
and
\begin{equation}\label{eq:def_c}
    \begin{split}
        \|\partial^{\alpha} h\|_{L^2(\Omega)}^2 = & \left(\frac{k!}{\left(k-|\alpha|\right)!}w^{\alpha}\right)^2\int_{\Omega} \sigma_{2(k-|\alpha|)}(w\cdot x+b) \di x\\
            \le & \left(\frac{k!}{\left(k-|\alpha|\right)!}\left(\diam(\Omega)+\max\{|c_1|,|c_2|\}\right)^{k-|\alpha|}\right)^2\int_{\Omega}1\di x\\
            := & c(\Omega,d,\alpha,k,c_1,c_2).
    \end{split}
\end{equation}
Then, for any non-negative integer $m\le k$, the following statement holds
$$
\|h\|_{H^{m}(\Omega)}\le \sqrt{\sum_{|\alpha|\leq m} c(\Omega,d,\alpha,k,c_1,c_2) } =: \tilde{c}(\Omega,d,m,k,c_1,c_2).
$$

We hereby establish the embedding inequality for the variation norm, which enhances the result presented in \cite[Lemma 1]{SX23} by explicitly quantifying the impact of the dimensionality index $d$.
\begin{lemma}\label{lem:var_sobolev}
    Given any fixed integers $d$, $k$ and the domain $\Omega = (0,1)^d$, there exists a constant $\tilde{c}(d,m,k)$
    such that for any $f\in \K(\PP_k)$ the following inequality holds true
    $$\| f \|_{ H^{m} (\Omega) }\le \tilde{c}(d,m,k) \|f\|_{\PP_k},\quad 0\le m\le k,$$
    where it follows that $\tilde{c}(d,m,k) \sim d^{\frac{k}{2}}$ asymptotically as $d\to \infty$.
\end{lemma}
\begin{proof}
    The generic inequality can be found in \cite[Lemma 1]{SX23}.
    Here, we specifically focus on the scenario where the domain $\Omega=(0,1)^d$.
    By making the selections $-c_1=c_2=\sqrt{d}$ and considering the diameter of the domain as $\diam(\Omega)=\sqrt{d}$, we derive
    $$c(d,\alpha,k)=\left(\frac{k!}{\left(k-|\alpha|\right)!}\right)^2\left(2\sqrt{d}\right)^{2(k-|\alpha|)},$$
    where the $c(d,\alpha,k):= c(\Omega,d,\alpha,k,c_1,c_2)$ is introduced in \eqref{eq:def_c}.
    Thus, we can estimate
    \begin{equation*}
        \begin{split}
            \sum_{|\alpha|\le m} c(d,\alpha,k)
                = & \sum_{s=0}^{m} \left( \sum_{|\alpha|=s} \left(\frac{k!}{\left(k-|\alpha|\right)!}\right)^2 \left(2\sqrt{d}\right)^{2(k-|\alpha|)} \right) \\
                = & \sum_{s=0}^{m} \frac{\left(s+d-1\right)!}{(d-1)!s!} \left(\frac{k!}{\left(k-s\right)!}\right)^2\left(2\sqrt{d}\right)^{2(k-s)}
                =:  \sum_{s=0}^{m} I_s.
        \end{split}
    \end{equation*}
    Here we can verify that
    $
    I_{s+1} = \frac{(s+d)(k-s)^2}{4(s+1)d}I_s,\quad 0\le s<k.
    $
    Notice that
    $$
    \frac{k+d-1}{4kd} \le \frac{\left(s+d\right)(k-s)^2}{4(s+1)d}\le \frac{k^2}{4}.
    $$
    Therefore, the ratio $\frac{I_{s+1}}{I_s}$ has an finite upper bound as $d\to \infty$.
    Given that $I_0=\left(2\sqrt{d}\right)^{2k}$,
    we can deduce that $\sum_{|\alpha|\le m} c(d,\alpha,k) \sim d^{k}$ asymptotically as $d\to\infty$ and denote $\tilde{c}(d,m,k):=\sqrt{\sum_{|\alpha|\le m} c(d,\alpha,k)}$.

    Referring to \cite[Lemma 3]{SX23}, for any $f\in \K(\PP_k)$, there exist a Borel measure $\mu$ on $\PP_k$ and an arbitrary small $\epsilon>0$, such that
    $f = \int_{h\in\PP_k}h\di \mu(h),$
    and $\|\mu\|_{\M(\PP_k)}<(1+\epsilon)\|f\|_{\PP_{k}}$.
    Therefore, we obtain
    $$ \|f\|_{H^{m}(\Omega)}\le (1+\epsilon)\sup_{h\in\PP_{k}}\|h\|_{H^{m}(\Omega)}\|f\|_{\PP_k}.$$
    Using the above inequality and $\sup_{h\in\PP_k} \|h\|_{H^{m}(\Omega)}\le \tilde{c}(d,m,k)$, the proof is concluded.
\end{proof}

Utilize the Lemma \ref{lem:var_sobolev} along with \eqref{eq:var_RBV}, we derive the embedding property of Radon-BV spaces.
\begin{corollary}\label{cor:RBV_sobolev}
    Let $\Omega=(0,1)^{d}$.
    For any $f\in \F_{k+1}(\Omega) \cap L^2(\Omega)$, the following holds
$$
\|f\|_{H^m\left(\Omega\right)} \le \tilde{c}(d,m,k) \frac{1}{k!}\left( 1+ 2^{k}M(d)d^{\frac{k}{2}} +k!M(d)\right)\left(\|f\|_{L^2(\Omega)}+ |f|_{\F_{k+1}(\Omega)} \right),
$$
as $d\to \infty$. Here, $M(d)$ represents the smallest value that satisfies the inequality
$\|p\|_{\PP_{k}} \le M(d) \|p\|_{L^2(\Omega)}$ for all $p\in \Poly_k$.
\end{corollary}

To further elaborate on the impact of the dimensionality index $d$ towards the Radon-BV seminorm, we derive a lower bound for $M(d)$ below when the domain is defined as $\Omega=(0,1)^d$.
\begin{lemma}\label{lem:M_lowerbound}
    For any fixed $k$, let $M(d)$ denote the smallest value that satisfies the inequality $\|p\|_{\PP_k}\le M(d)\|p\|_{L^2(\Omega)}$ for all $p$ in the polynomial space with a degree strictly less than $k+1$.
    Let $\Omega=(0,1)^d$, and $-c_1=c_2=\sqrt{d}$ for the dictionary $\PP_k$. Then, it holds that $M(d)\gtrsim  d^{-\frac{k}{2}}$ asymptotically.
\end{lemma}
\begin{proof}

    According to Lemma \ref{lem:var_sobolev} and (\ref{eq:var_RBV}), we have
    $$\|p\|_{H^{m}(\Omega)} \le \tilde{c}(d,m,k) \|p\|_{\PP_{k}}\le \tilde{c}(d,m,k) M(d)\|p\|_{L^2(\Omega)},$$
    for all $p\in \Poly_k$ with $\tilde{c}(d,m,k) \sim d^{\frac{k}{2}}$ asymptotically since $|p|_{\F_{k+1}(\Omega)}=0$.

    Let us consider the spacial polynomial $p(x) = \left(x_1-\frac{1}{2}\right)^k$ for $\Omega=(0,1)^d$.
    This polynomial only depends on one variable $x_1$.
    Given a multi-index $\alpha$ with $|\alpha|\le k$, we observe that if $\alpha_i>0$ for some index $i > 1$, then the partial derivative $\partial^{\alpha}p$ vanishes, i.e., $\partial^{\alpha}p=0$.
    Otherwise,
    \begin{equation*}
        \begin{split}
        \|\partial_{x_1}^{\alpha_1}p\|_{L^2(\Omega)}^2
        & = \int_{0}^{1} \left(\frac{k!}{\left(k-\alpha_1\right)!}\left(x_1-\frac{1}{2}\right)^{k-\alpha_1}\right)^2\di x_1
        \end{split} \sim 1,
    \end{equation*}
    when $d\to \infty$.
    It is also true that $\|p\|_{H^{m}(\Omega)}\sim 1$ when $d\to\infty$.
    Therefore, in the limit as  $d\to \infty$, we have
    $M(d)\gtrsim \tilde{c}(d,m,k)^{-1} \sim d^{-\frac{k}{2}}$.
\end{proof}
\begin{remark}
    We note that the rate in Lemma \ref{lem:M_lowerbound} may not be the sharpest possible. Determining the precise sharp lower bound as well as the upper bound for $M(d)$ remains an unresolved problem.
\end{remark}

\subsection{Summary of the relationship among different norms }

The main result in this section can be summarised as the following theorem.
\begin{theorem}\label{thm:sobolev_3spaces}
    Let $\Omega=(0,1)^{d}$,
    \begin{enumerate}[label=(\alph*)]
        \item\label{enm:sobolev_barran} if the function $f$ belongs to an extended Barron space i.e. $f\in B_1^{k}(\Omega)$, there hold
        $$\|f\|_{H^{m}(\Omega)}\le C(d,m,k)\|f\|_{B_{1}^{k}(\Omega)},\quad 0\le m\le k,$$
        and $C(d,m,k)\sim d^{\frac{m}{2}}$ asymptotically as $d\to \infty$;

        \item\label{enm:sobolev_var} if the function $f$ belongs to a variation space i.e. $f\in\K(\PP_{k})$ with $-c_1=c_2=\sqrt{d}$ for $\PP_{k}$, there hold
        $$\|f\|_{H^{m}(\Omega)}\le \tilde{c}(d,m,k)\|f\|_{\PP_{k}},\quad 0\le m\le k,$$
        and $\tilde{c}(d,m,k)\sim d^{\frac{k}{2}}$ asymptotically as $d\to \infty$;

        \item\label{enm:sobolev_RBV} if the function $f$ belongs to a Radon-BV space and is square-integrable i.e. $f\in \F_{k+1}(\Omega) \cap L^2(\Omega)$, there hold
        \begin{equation*}
            \begin{split}
                \|f\|_{H^{m}(\Omega)} \le &\tilde{c}( d,m,k)\frac{1}{k!}\left( 1+ 2^{k}M(d)d^{\frac{k}{2}} +k!M(d)\right)\\
                &\times \left(\|f\|_{L^{2}(\Omega)} + |f|_{\F_{k+1}(\Omega)}\right),
            \end{split}
        \end{equation*}
        for $0\le m\le k$, with the same constant $\tilde{c}(d,m,k)$ as in the above item (b),
        and $M(d)\gtrsim d^{-\frac{k}{2}}$ asymptotically as $d\to \infty$.
    \end{enumerate}
\end{theorem}

\begin{proof}
The proof is detailed in the above subsections. Specifically, item \ref{enm:sobolev_barran} is established utilizing Lemma \ref{lem_barronsobolev}. Furthermore, the item \ref{enm:sobolev_var} is derived based on Lemma \ref{lem:var_sobolev}. Lastly, item \ref{enm:sobolev_RBV} is proven by virtue of Corollary \ref{cor:RBV_sobolev} and Lemma \ref{lem:M_lowerbound}.
\end{proof}

\section{Revisit of the approximation property}\label{se_revisitUA}
To investigate error bounds associated with the regularization scheme (\ref{eq:min_problem}) utilizing different penalty terms, it is imperative to implement interpolation techniques across various Sobolev spaces. To this end, revisiting the approximation property in the Hilbert space $L^2(\Omega)$ is crucial for our analysis. The subsequent approximation theories take into account the dimensionality of the function domains explicitly.

We first present the approximation result for the extended Barron spaces, which does not depend on the dimensionality index $d$ for the specific domain $\Omega$ and enhance the existing result presented in \cite[Theorem 8]{LLMP23}.
\begin{lemma}\label{lem:approx_Barron}
    Let $\Omega=(0,1)^{d}$ and $n\in \N$. Assume that $f$ belongs to the extended Barron space $B_1^{k}(\Omega)$ and for any $\epsilon$ there exist a probability distribution $\rho_{\epsilon}$ such that $\|f\|_{1,\rho_{\epsilon}}^k $ in Definition \ref{def:extendedBarron} satisfies $\|f\|_{B_{1,\rho_{\epsilon}}^k} \leq (1+\epsilon) \|f\|_{B_1^{k}(\Omega)}$. Then there is a two-layer neural network $f_n=\frac{1}{n}\sum_{i=1}^{n}a_i\sigma_{k}\left( w_i\cdot x + b_i \right)$, yielding
    \begin{align}\label{eq:barronUni}
    \|f-f_n\|_{L^2(\Omega)} \le C(k)(1+\epsilon)n^{-\frac{1}{2}}\|f\|_{B_1^{k}(\Omega)},
    \end{align}
    and $\frac{1}{n}\sum_{i=1}^{n}|a_i|\left(
    \|w_i\|_1+|b_i|\right)^{k}\le (1+\epsilon)\|f\|_{B_{1}^{k}(\Omega)}$.
\end{lemma}
\begin{proof}
The proof relies on that of  \cite[Theorem 8]{LLMP23}. It is crucial to recall the random vectors $\left\{g_i(\cdot)\right\}_{i=1}^{n}$ and distribution $\bar{\rho}$ as described in the proof of \cite[Theorem 8]{LLMP23} such that
    \begin{equation*}
        \begin{split}
           \EE_{\bar{\rho}} \|f-f_n\|_{L^2(\Omega)}^2 \le & \frac{1}{n} \EE_{\bar{\rho}}\|g_1\|_{L^{2}(\Omega)}^{2}
           \overset{(\ref{eq:L2_B1k})}{\le} C(k)^2\frac{1}{n}\EE_{\rho}\|g_1\|_{B_1^{k}(\Omega)}^{2}
            \\
            \le & C(k)^2\frac{1}{n}(1+\epsilon)^2\|f\|_{B_1^{k}(\Omega)},
        \end{split}
    \end{equation*}
    where we have implemented the refined estimate (\ref{eq:L2_B1k}). Then the rest of the arguments hold as in the proof of \cite[Theorem 8]{LLMP23}.
\end{proof}

We shall emphasize that the constant $C(k)$ in (\ref{eq:barronUni}) is independent of the dimensionality index $d$ for $f\in B_1^{k}(\Omega)$. This independence is a significant confirmation of overcoming the curse of dimensionality in high-dimensional settings. However, it is worth noting that this behavior does not extend to the other two spaces, as will be demonstrated in the subsequent results.

\begin{lemma}\label{lem:approx_var_L2}
    Let $\Omega=(0,1)^d$, and $n\in\N$. Denote the dictionary $\PP_k$ as in Definition \ref{def:variation_space} with $-c_1=c_2=\sqrt{d}$.
    Assume $f$ belongs to the variation space $\K(\PP_k)$.
    Then for any $\epsilon>0$, there is a two-layer neural network $f_n=\sum_{i=1}^{n}a_i\sigma_{k}\left(
    w_i+b_i\right)$ satisfying $\|w_i\|_2=1$, $|b_i|\le\sqrt{d}$,
    $$ \|f-f_n\|_{L^2(\Omega)} \le (1+\epsilon)2^{k} d^{\frac{k}{2}}n^{-\frac{1}{2}}\|f\|_{\PP_k},$$
    and $\sum_{i=1}^{n}|a_i|=(1+\epsilon)\|f\|_{\PP_k}$.
\end{lemma}
\begin{proof}
    According to \cite[Lemma 3]{SX23},
    there is a measure $\mu$ on $\PP_k$ such that $f=\int_{\PP_k}i_{\PP_k\to L^2(\Omega)}\di\mu$ where $i_{\PP_k\to L^2(\Omega)}$ is the identity map from $\PP_k$ to $L^2(\Omega)$, and the measure $\|\mu\|_{\M(\PP_k)}\le (1+\epsilon)\|f\|_{\PP_k}$.
    Decomposite $\mu=\mu_+-\mu_-$ with $\mu_+\ge 0$ and $\mu_-\ge 0$.
    Define $\bar{\mu}$ as a probability measure on $\{\pm 1\}\times \PP_{k}$ by
    $$ \bar{\mu}\left(A\right)=\|\mu\|_{\M(\PP_k)}^{-1} \left(\mu_+\left(\left\{ g\in\PP_{k}:\left(1,g\right)\in A    \right\}\right)+\mu_-\left(\left\{ g\in\PP_{k}:\left(-1,g\right)\in A   \right\}\right)  \right),$$
    for any measurable set $A\subset \{\pm 1\}\times \PP_k$.
    Then, we can rewrite the integral form of $f$ as
    $$f=\|\mu\|_{\M(\PP_{k})}\int_{\{\pm 1\}\times \PP_k} \nu g \di\bar{\mu}(\nu,g).$$
    Assume $\left\{(\nu_i,g_i)\right\}_{i=1}^{n}$ are \textit{i.i.d.} random variables with probability distribution $\bar{\mu}$ and $f_n=\|\mu\|_{\M(\PP_{k})}n^{-1}\sum_{i=1}^{n}\nu_ig_i$.
   % Since the independence holds among $\{(\nu_i,g_i)\}_{i=1}^{n}$, the expectancy
%    \begin{equation*}
%        \begin{split}
%            \EE\left[ \left<f-\|\mu\|_{\M(\PP_k)}\nu_ig_i,f-\|\mu\|_{\M(\PP_k)}\nu_jg_j \right> \right]=0,\quad i\neq j.
%        \end{split}
%    \end{equation*}
    Therefore,
    \begin{equation*}
        \begin{split}
            \EE \left[ \|f-f_n\|_{L^2(\Omega)}^2 \right]
            = & \EE \left[ \left\| \frac{1}{n}\sum_{i=1}^{n}f-\frac{1}{n}\sum_{i=1}^{n}\|\mu\|_{\M(\PP_{k})}\nu_ig_i \right\|_{L^2(\Omega)}^{2} \right]\\
            = & \frac{1}{n}\EE \left[\left\| f-\|\mu\|_{\M(\PP_{k})}\nu_1g_1\right\|_{L^2(\Omega)}^{2} \right].
        \end{split}
    \end{equation*}
    Besides, we have
    \begin{equation*}
        \begin{split}
            \EE\left[ \left\|\|\mu\|_{\M(\PP_{k})}\nu_1g_1 \right\|_{L^2(\Omega)}^2\right]
            = & \EE\left[ \left\| \left(\|\mu\|_{\M(\PP_{k})}\nu_1g_1 - f \right) + f \right\|_{L^2(\Omega)}^2\right]\\
            = & \EE\left[ \left\|\|\mu\|_{\M(\PP_{k})}\nu_1g_1-f\right\|_{L^2(\Omega)}^2 \right] + \|f\|_{L^2(\Omega)}^2 \\
            & \quad+ 2 \EE\left[ \left<\|\mu\|_{\M(\PP_{k})}\nu_1 g_1-f ,f\right>_{L^{2}(\Omega)} \right]
        \end{split}
    \end{equation*}
    and $\EE\left[ \left<\|\mu\|_{\M(\PP_{k})}\nu_1 g_1-f ,f\right>_{L^{2}(\Omega)} \right]=0$. Thus,
    \begin{equation}\label{eq:MC_varspaces}
        \begin{split}
            \EE \left[ \|f-f_n\|_{L^2(\Omega)}^2 \right]
            = & \frac{1}{n}\EE\left[ \left\|\|\mu\|_{\M(\PP_{k})}\nu_1g_1\right\|_{L^2(\Omega)}^2 \right]-\frac{1}{n}\|f\|_{L^2(\Omega)}^2\\
            \le & \frac{1}{n}\|\mu\|_{\M(\PP_{k})}^2\EE\left[ \|\nu_1g_1\|_{L^{2}(\Omega)}^2 \right]
            \le \frac{1}{n}\|\mu\|_{\M(\PP_{k})}^2\sup_{g\in \PP_{k}}\|g\|_{L^2(\Omega)}^{2}.
        \end{split}
    \end{equation}
%    For $w\in \Sf^{d}$ and $|b|\le \sqrt{d}$,
%    \begin{equation*}
%        \begin{split}
%            \int_{\Omega} \left(\sigma_{k}\left(w\cdot x + b\right)\right)^2 \di x
%            \le \int_{\Omega} \left(\|w\|_2^2+1\right)^{k}\left(\|x\|_2^2+b^2\right)^{k} \di x \le 2^{2k}d^{k}.
%        \end{split}
%    \end{equation*}
    Thus, we again use $\sup_{g\in \PP_{k}}\|g\|_{L^2(\Omega)}^2\le 2^{2k}d^{k}$ in (\ref{eq:norm_g}) and, together with \eqref{eq:MC_varspaces}, we can find a realization of $f_n$ satisfying Lemma \ref{lem:approx_var_L2}.
\end{proof}
\begin{remark}
In both Lemmas \ref{lem:approx_Barron} and \ref{lem:approx_var_L2}, the convergence rate $n^{-\frac{1}{2}}$ with respect to the neuron number $n$ is not sharp. For instance,
    \cite[Theorem 9]{LLMP23} demonstrates a convergence rate of $\|f-f_n\|_{L^2(\Omega)}\lesssim n^{-\frac{1}{2}-\frac{1}{d}}$ for $f\in B_1^{k}$. Similarly,  \cite[Theorem 3]{SX22} establishes a rate of $\|f-f_n\|_{L^{2}(\Omega)}\lesssim n^{-\frac{1}{2}-\frac{2k+1}{2d}}$ for $f\in\K(\PP_k)$.
    However, it is noteworthy that these rates involve implied constants that are dependent on the dimensionality index $d$. More precisely, these constants are influenced by the diameter of the partitions within the bounded domain $\Omega$.

    We provide a brief explanation of the dependency of the constants on the dimensionality index $d$. Consider a bounded manifold in $\mathbb{R}^d$. When we divide this manifold into $n$ relatively equal-sized portions, the diameter of each portion scales as
    $\mathcal{O}(n^{\frac{1}{d}})$. However, for a fixed $n$, there exists a positive lower bound on the diameter of these portions, which is related to the dimensionality. Unfortunately, estimating the exact relationship between this lower bound and $d$ can be challenging. It is typically influenced by the geometry and properties of the manifold which is out of the scope of current work.
\end{remark}

\section{Regularization and error bound analysis}\label{se_error}
In this section, we conduct a comprehensive error bound analysis for the general regularization schemes (\ref{eq:min_problem}), considering various penalty terms.

We will frequently utilize the interpolation inequality, as elaborated in Theorem 5.2 of \cite{Adams_book}, throughout this section. It is worth noting that the aforementioned reference does not account for the dimensionality dependence. Therefore, we present the following lemma to emphasize the impact of the dimensionality index $d$ toward the interpolation inequality.

\begin{lemma}\label{lem:interpolation}
Assume that the domain $\Omega \subset \mathbb{R}^{d}$ satisfies the cone condition, meaning that there exists a cone 
\[
C_o = \left\{ x \in \mathbb{R}^{d} \mid x = 0 \textrm{ or } \left(0 \leq \|x\|_2 \leq \rho \textrm{ and } \angle(x, v) \leq \frac{\kappa}{2}\right) \right\},
\]
where $0$ is the vertex, $v \in \mathbb{R}^{d}$ is the axis direction, $\rho > 0$ is the height, and $0 < \kappa \leq \pi$ is the aperture angle. The cone condition stipulates that for every point $x \in \Omega$, there exists another cone $C_{x}$ with vertex at $x$, contained entirely within $\Omega$, and congruent to $C_o$.

Let $d, k \in \mathbb{N}$ with $0 \leq m \leq k$. For every $u \in W^{k,p}(\Omega)$, $p\in [1,\infty)$, the following interpolation inequality holds:
\begin{equation}\label{eq:interpolation}
    \|u\|_{W^{m,p}(\Omega)} \leq K(d, k, m,p) \|u\|_{W^{k,p}(\Omega)}^{\frac{m}{k}} \|u\|_{L^{p}(\Omega)}^{1 - \frac{m}{k}},
\end{equation}
where $K(d, k, m,p) > 0$ is a constant that depends on $d$, $k$, $m$ and $p$. Specifically, there exists a relationship such that
\begin{equation}\label{eq:K_dkm}
    K(d, k, m, p) \sim d^{c(k, m, p)},
\end{equation}
with $c(k, 0, p) = 0$ and a finite $c(k, m , p) > 0$ for $1 \leq m \leq k$, asymptotically as $d \to \infty$. 
\end{lemma}
\begin{proof}
If $m=0$, the result is trivial. For $1 \leq m \leq k$,
the key ingredient of the proof lies in tracing the dependence of the dimensionality index through various Sobolev norms. Below, we prove a simplest case of (\ref{eq:interpolation}) as an illustration for the seminorm such that
\begin{equation}\label{eq:inter_simple}
    \left|u\right|_{m,p}\le K^{*}\left(\epsilon'\left|u\right|_{k,p} + \epsilon'^{-\frac{m}{k-m}}\left\|u\right\|_{L^{p}}\right).
\end{equation}

To this end, we assume that $u\in C^{\infty}(\Omega)$. Denote $\Sigma=\left\{\beta\in \R^{d}\mid \left\|\beta\right\|_2=1\right\}$. For every $x\in \Omega$, let $C_x$ be the cone that satisfies the cone condition with vertex $x$, axis direction $v_x$, height $\rho_0$, and aperture angle $\kappa$ and $\Sigma_{x} = \left\{\beta \in \Sigma\mid \angle\left(\beta,v_x\right)\le \frac{\kappa}{2}\right\}$. 
Without loss of generality, we let $\beta \in \Sigma_{x}$. Then, according to \cite[Lemma 5.4]{Adams_book}, for every $0<\rho\le \rho_0$, there holds
\begin{equation}\label{eq:K_p}
    \left|\beta\cdot \grad u(x)\right|^{p}\le \frac{K_p}{\rho}I\left(\rho,p,u,x,\beta\right),
\end{equation}
where $K_{p}=2^{p-1}9^{p}$, and 
\[
    I\left(\rho,p,u,x,\beta\right)=\rho^{p}\int_{0}^{\rho}\left|D_{t}^{2}u\left(x+t\beta\right)\right|^{p}\di t + \rho^{-p}\int_{0}^{\rho}\left|u\left(x+t\beta\right)\right|^{p}\di t. 
\]
    
Furthermore, there is a cone $C_{x,u}$ with vertex $x$, axis direction $v_{x,u}$, and aperture angle $\frac{\kappa}{3}$, and the $\Sigma_{x,u} = \left\{\beta\in \Sigma\mid \angle\left(\beta,v_{x,u}\right)\le \frac{\kappa}{6}\right\}\subset \Sigma_{x}$. In particular, there holds $\angle\left(\beta, \grad u(x)\right)\le \frac{\pi}{2}-\frac{\kappa}{6}$ for every $\beta \in \Sigma_{x,u}$. Let $K_1 = \cos^{p}\left(\frac{\pi}{2}-\frac{\kappa}{6}\right)\int_{\Sigma_{x,u}}\di \beta$. Then 
    \begin{equation}\label{eq:K_1}
        \int_{\Sigma_{x,u}}\left|\beta\cdot \grad u(x)\right|^{d}\di \beta \ge K_1\left|\grad u(x)\right|^{p}.
    \end{equation}
We combine \eqref{eq:K_p} and \eqref{eq:K_1} to obtain 
    \begin{equation*}
        \begin{split}
            \int_{\Omega}\left|\grad u(x)\right|^{p}\di x \le & \int_{\Omega} \frac{1}{K_1}\int_{\Sigma_{x,u}}\left|\beta\cdot \grad u(x)\right|^{p}\di \beta \di x\\
            \le & \frac{K_{p}}{\rho K_1}\int_{\Sigma_{x,u}}\int_{\Omega}I\left(\rho,p,u,x,\beta\right)\di x \di \beta.
        \end{split}
    \end{equation*}
When we fix $\beta = e_d=\begin{bmatrix}
        0& \cdots & 0 & 1
    \end{bmatrix}$, it has been proved in \cite{Adams_book} that  
    \[
    \int_{\Omega}I\left(\rho,p,u,x,e_d\right)\di x \le \rho\int_{\Omega} \left(\rho^{p}\left|D_{d}^{2}u(x)\right|^{p}+\rho^{-p}\left|u(x)\right|^{p}\right)\di x.
    \]
    Then, for every $\beta\in \Sigma$, there holds
    \begin{equation*}
        \begin{split}
            \int_{\Omega}I\left(\rho,p,u,x,\beta\right)\di x \le & \rho\int_{\Omega}\left(\rho^{p}\left|\sum_{j=1}^{d}\sum_{i=1}^{d}\beta_i\beta_j D_{j,i}^{2}u(x)\right|^{p} + \rho^{-p}\left|u(x)\right|^{p}\right)\di x\\
            \le & \rho\int_{\Omega} \left(\rho^{p}\left|2\sum_{|\alpha|=2}D^{\alpha}u(x)\right|^p + \rho^{-p}\left|u(x)\right|^{p}\right) \di x\\
            \le & \rho \int_{\Omega}\left( 2^{p}\left(\frac{(d+1)d}{2}\right)^{p}\sum_{|\alpha|=2}\left|D^{\alpha}u(x)\right|^{p} + \rho^{-p}\left|u(x)\right|^{p} \right) \di x\\
            \le & (d+1)^{p}d^{p}\rho\left(\rho^{p}\left|u\right|_{2,p}^{p} + \rho^{-p}\left\|u\right\|_{L^{p}}^{p}\right),
        \end{split}
    \end{equation*}
    where the third line is due to $\left|\sum_{n=1}^{N}a_n\right|^{p}\le N^{p}\sum_{n=1}^{N}\left|a_n\right|^{p}$ when $N\in\N$ and $a_n\in \R$ for every $1\le n\le N$.

Next, considering that $\left|D_{j}(u)\right|\le \left|\grad u\right|$ for every index $j=1,\cdots,d$, we have 
    \[
    \left|u\right|^p_{1,p}\le \frac{K_p(d+1)^{p}d^{p+1}}{\cos^{p}\left(\frac{\pi}{2}-\frac{\kappa}{6}\right)}\left(\rho^{p}\left|u\right|_{2,p}^{p}+\rho^{-p}\left\|u\right\|_{L^{p}}^{p}\right).
    \]
Finally, with the density of $C^{\infty}(\Omega)$ in $W^{2,p}(\Omega)$, we have 
    \begin{equation}\label{eq:inter_m1k2}
        \left|u\right|_{1,p}\le K_2\left(\rho\left|u\right|_{2,p} + \rho^{-1}\left\|u\right\|_{L^{p}}\right),
    \end{equation}
where $K_2=\mathcal{O}\left(d^{2+\frac{1}{p}}\right)$. 

One then can complete the proof along the line of \cite[Section 5.2]{Adams_book} by induction. Though it is hard to explicitly calculate the $c(k,m,p)$, one can show that such a value is finite since $k$ is finite and $p\in [1,\infty)$.
\end{proof}

In the rest of this section, we will frequently use the notations $K(d,k,m):=K(d,k,m,2)$ and $c(k,m) := c(k,m,2)$ to denote the influence induced by the interpolation. It is worth mentioning that if we consider the domain to be a torus, then the constant $K(d,k,m,p)$ simplifies to 1, as can be observed through Fourier analysis. We omit the details here for brevity.

\subsection{Regularization with the extended Barron norm penalty}
Extensive research on regularization using the extended Barron norm penalty has been conducted in \cite{LLMP23}. By carefully calibrating the dependency on the dimensionality index $d$ and recalling the Tikhonov functional
\begin{align}\label{eq:Tik_extendedRadon}
   J_{\lambda}(g): = \left\|g-f^{\delta}\right\|_{L^2(\Omega)}^2 + \lambda \left( \frac{1}{n}\sum_{i=1}^{n}|a_i|\left(\|w_i\|_1+|b_i|\right)^{k} \right)^{2}, \qquad g\in F_n,
\end{align}
with $F_n=\left\{\frac{1}{n}\sum_{i=1}^{n}a_i\sigma_{k}\left(w_i\cdot x + b_i\right)\mid a_i\in\R, w_i\in \R^d, b_i\in \R\right\}$,
we introduce the refined error bound presented below.

\begin{theorem}\label{thm:Barron_regularization}
    Let $\Omega=(0,1)^{d}$, and $f\in B_1^{k}(\Omega)$.
    Recall the noise level $\delta$ in (\ref{eq:noisymeasurement}).
    Denote $f_{n,\lambda}^{\delta}$ as the minimizer of the Tikhonov functional (\ref{eq:Tik_extendedRadon}) and choose the regularization parameter $\lambda$ as $\delta + C(k) n^{-\frac{1}{2}} \|f\|_{B_1^k(\Omega)} =  \sqrt{\lambda} \|f\|_{B_1^k(\Omega)}$ with the constant $C(k)$ from Lemma \ref{lem:approx_Barron},
    there holds
    $$
    \|f_{n,\lambda}^{\delta}-f\|_{H^{m}(\Omega)}= \mathcal{O}(d^{\frac{m}{2}+c(k,m)}(\delta+n^{-\frac{1}{2}})^{\frac{k-m}{k}}),
    $$
    as $d\rightarrow \infty$.
    The implied constant is only related to $k$, $m$ and $\|f\|_{B_{1}^{k}(\Omega)}$.
\end{theorem}
\begin{proof}
    We denote $f_{n,\lambda}^{\delta} = \frac{1}{n}\sum_{i=1}^{n}a_i^{*}\sigma_{k}\left(w_i^{*}\cdot x + b_{i}^{*}\right)$ and define $f_n$ to be identical to the one in Lemma \ref{lem:approx_Barron}. It holds that
    \begin{equation*}
        \begin{split}
            \left\| f^{\delta}-f_{n,\lambda}^{\delta} \right\|_{L^2(\Omega)}
            & \le \|f^{\delta} - f\|_{L^{2}(\Omega)} + \|f-f_n\|_{L^2(\Omega)} + \frac{\sqrt{\lambda}}{n}\sum_{i=1}^{n}|a_i|\left( \|w_i\|_1 + |b_i| \right)^{k} \\
            & \le \delta + (1+\epsilon)C(k)n^{-\frac{1}{2}}\|f\|_{B_{1}^{k}(\Omega)} + (1+\epsilon)\sqrt{\lambda}\|f\|_{B_1^{k}(\Omega)}\\
            & \le 2 (1+\epsilon)\left(\delta + C(k)n^{-\frac{1}{2}}\|f\|_{B_{1}^{k}(\Omega)} \right),
        \end{split}
    \end{equation*}
    since $\delta + C(k) n^{-\frac{1}{2}} \|f\|_{B_1^k} = \sqrt{\lambda} \|f\|_{B_1^k}$. Thus, we derive
    \[
    \|f-f_{n,\lambda}^{\delta}\|_{L^2(\Omega)}\le \|f-f^{\delta}\|_{L^2(\Omega)} + \left\| f^{\delta}-f_{n,\lambda}^{\delta} \right\|_{L^2(\Omega)}\le 3(1+\epsilon)\left(\delta + C(k)n^{-\frac{1}{2}}\|f\|_{B_{1}^{k}(\Omega)} \right).
    \]
    Similarly, we derive
    \begin{equation*}
        \begin{split}
            \|f_{n,\lambda}^{\delta}\|_{B_1^{k}(\Omega)}& \le \frac{1}{n}\sum_{i=1}^{n}|a_{i}^{*}|\left( \|w_i^*\|_1+|b_i^*| \right)^{k}\\
            & \le 2 \frac{(1+\epsilon)\left(\delta + C(k)n^{-\frac{1}{2}}\|f\|_{B_{1}^{k}(\Omega)} \right)}{\sqrt{\lambda}} \le 2(1+\epsilon) \|f\|_{B_{1}^{k}(\Omega)}.
        \end{split}
    \end{equation*}
    Furthermore, by \ref{enm:sobolev_barran} in Theorem \ref{thm:sobolev_3spaces}, we can further deduce that
    \begin{equation*}
        \begin{split}
            \|f-f_{n,\lambda}^{\delta}\|_{H^{k}(\Omega)}
            & \le \|f\|_{H^{k}(\Omega)} + \|f_{n,\lambda}^{\delta}\|_{H^{k}(\Omega)} \le C(d,k,k)\left(\|f\|_{B_{1}^{k}(\Omega)}+\|f_{n,\lambda}^{\delta}\|_{B_{1}^{k}(\Omega)}\right)\\
            & \le 3(1+\epsilon)C(d,k,k)\|f\|_{B_{1}^{k}(\Omega)}.
        \end{split}
    \end{equation*}
    By applying the interpolation \eqref{eq:interpolation}, we obtain
    \begin{equation*}
        \begin{split}
            \|f-f_{n,\lambda}^{\delta}\|_{H^{m}(\Omega)}
            \le &K(d,k,m)\|f-f_{n,\lambda}^{\delta}\|_{H^{k}(\Omega)}^{\frac{m}{k}} \|f-f_{n,\lambda}^{\delta}\|_{L^{2}(\Omega)}^{\frac{k-m}{k}}\\
            \le &3 (1+\epsilon)K(d,k,m)\left(C(d,k,k)\right)^{\frac{m}{k}} \\
            &\times \left(\delta + C(k)n^{-\frac{1}{2}}\|f\|_{B_{1}^{k}(\Omega)}\right)^{\frac{k-m}{k}} \|f\|_{B_{1}^{k}(\Omega)}.
        \end{split}
    \end{equation*}
    Observing that $C(d,k,k)\sim d^{\frac{k}{2}}$ when $d\to \infty$, \eqref{eq:K_dkm}, and the arbitrariness of $\epsilon$, we have proven this theorem.
\end{proof}

\subsection{Regularization with the variation norm penalty}

Similarly, we can conduct an error-bound analysis for the regularization scheme (\ref{eq:min_problem}), incorporating a variation norm penalty.
\begin{theorem}\label{thm:vs_regularization}
    Let $\Omega=(0,1)^{d}$ and $\PP_k$ be the dictionary in Definition \ref{def:variation_space} with $-c_1=c_2=\sqrt{d}$. Assume that $f\in\K(\PP_k)$ and $\|f-f^\delta\|_{L^2(\Omega)}\le \delta$.
    Consider the Tikhonov functional
    $$\tilde{J}_{\lambda}(g):=\|g-f^{\delta}\|_{L^2(\Omega)}^2 + \lambda\left(
\sum_{j=1}^{n}|a_j| \right)^2, \quad g\in\tilde{F}_n,$$
    with $\tilde{F}_n=\left\{ \sum_{j=1}^{n}a_jh_j\mid a_j\in\R, h_j\in \PP_k \right\}$.
    There exists a minimizer $f^{\delta,\lambda,n}$ of $\tilde{J}_{\lambda}$ with a regularization parameter $\lambda$ as $\delta + 2^k d ^{\frac{k}{2}}n^{-\frac{1}{2}}\|f\|_{\PP_k} = \sqrt{\lambda}\|f\|_{\PP_k}$ that satisfies
    % $$
    % \left\{a_j^*,h_j^*\right\}_{j=1}^{n}=\mathop{\arg\min}\limits_{\left\{a_j,h_j\right\}_{j=1}^{n},a_j\in\R,h_j\in\PP_k}\left\|f^\delta-\sum_{j-1}^{n}a_jh_j\right\|_{L_2(\Omega)}^2 + \lambda \left(\sum_{j=1}^{n}|a_j|\right)^2,
    % $$
    % and define $f^{\delta,\lambda,n} = \sum_{j=1}^{n}a_j^*h_j^*$.
    % Then, there holds
    the error bound
    \[\|f-f^{\delta,\lambda,n}\|_{H^m(\Omega)} = \mathcal{O}\left( d^{\frac{m}{2}+c(k,m)}\left(\delta + d^{\frac{k}{2}}n^{-\frac{1}{2}}\right)^{\frac{k-m}{k}} \right),\]
    as $d\to \infty$. 
    The implied constant is only related to $k$, $m$ and $\|f\|_{\PP_{k}}$.
\end{theorem}

\begin{proof}
First, we show that the minimizer set is well-defined. Recall that $h_j(x) = \sigma_{k}\left(w_j\cdot x+b_j\right)$ is parameterized by $w_j\in\Sf^{d-1}$ and $b_j\in[c_1,c_2]$. Meanwhile, the coefficient $a_j$ is subject to the constraint $\sum_{j=1}^{n}|a_j|\le \lambda^{-\frac{1}{2}}\|f^{\delta}\|_{L_2(\Omega)}$ when $\lambda>0$.
    Thus, the minimizer set $\left\{a_j^*,h_j^*\right\}_{j=1}^{n}$ is reachable since the search space of $\{a_j,w_j,b_j\}_{j=1}^{n}$ is compact.

Then, similar to the proof of Theorem \ref{thm:Barron_regularization}, with the help of $f_n$ in Lemma \ref{lem:approx_var_L2}, we obtain following error bounds
    \begin{displaymath}
    \|f-f^{\delta,\lambda,n}\|_{L^{2}(\Omega)} \le 3(1+\epsilon)\left[\delta + 2^k d ^{\frac{k}{2}} n^{-\frac{1}{2}}\|f\|_{\PP_k}\right].
    \end{displaymath}
    Together with \ref{enm:sobolev_var} in Theorem \ref{thm:sobolev_3spaces}, we also have
    \[
        \|f-f^{\delta,\lambda,n}\|_{H^k(\Omega)} \le 3(1+\epsilon) \tilde{c}(d,k,k)\|f\|_{\PP_{k}}.
    \]
We thus end the proof by implementing the interpolation inequality and noticing $\tilde{c}(d,k,k)\sim d^{\frac{k}{2}}$ when $d\to \infty$,\eqref{eq:interpolation}, \eqref{eq:K_dkm}, and the arbitrariness of $\epsilon$.
\end{proof}

\subsection{Regularization with the Radon-BV seminorm penalty}

The third subsection is devoted to the analysis of the error bound when the penalty term is chosen as the Radon-BV seminorm, where
we make use of the approximation rate of finite neurons in the previous two subsections.
We shall mention that \cite{PN23} shows a representation theorem and an approximation rate for functions in Radon-BV spaces when $k=1$, which is related to $\relu$ activation functions.
%\cite{U23} has extended the representation theorem to cover other types of activation functions.
However, as far as we know, there is no approximation result for functions in Radon-BV spaces and their derivatives when $k> 1$.
So, compared to the preceding two subsections, the focus in this subsection shifts to identifying a minimizer for the following functional:
$$
L(g) := \|f^{\delta}-g\|_{L^2(\Omega)}^2 + \lambda |g|_{\F_{k+1}(\Omega)}^2.
$$
Since $\|\cdot\|_{L^2(\Omega)}+|\cdot|_{\F_{k+1(\Omega)}}$ is a norm on $\F_{k+1}(\Omega)$, the unit norm ball in $\F_{k+1}(\Omega)$ is weakly closed, as stated in \cite[Theorem 11.5]{Clason2020}. However, despite the fact that the functional $L(\cdot)$ is weakly lower semicontinuous, and the norm-bounded ball of $\F_{k+1}(\Omega)$ is both convex and closed, the existence of a minimizer for $L(\cdot)$ relies on the reflexivity of $\F_{k+1}(\Omega)$, referring to \cite[Theorem 11.10]{Clason2020}. Unfortunately, it remains unclear whether $\F_{k+1}(\Omega)$ is a reflexive space or not.

To overcome this limitation, we choose a small positive real number $\varepsilon > 0$ and seek a solution $f_{\lambda}^{\delta,\varepsilon}$ satisfying
\begin{align}\label{eq:Tik_RadonBV}
L(f_{\lambda}^{\delta,\varepsilon}) \le L(g)+ \varepsilon,\quad \forall g\in \F_{k+1}(\Omega).
\end{align}
The error bound analysis is summarized below.

\begin{theorem}\label{thm:RadonBV_regularization}
    Let $\Omega = (0,1)^{d}$. Assume that $f\in \F_{k+1}(\Omega) \cap L^2(\Omega)$, with $\|f-f^{\delta}\|_{L^2(\Omega)}\le \delta$ and $f_{\lambda}^{\delta,\varepsilon} \in \F_{k+1}(\Omega)$ satisfies (\ref{eq:Tik_RadonBV}). Choose the regularization parameter
    $\lambda$ as
    $ \sqrt{\lambda} |f|_{\F_{k+1}(\Omega)} = \delta + \sqrt{ \varepsilon}$.
    Then, there holds
    $$
    \|f-f_{\lambda}^{\delta,\varepsilon}\|_{H^{m}(\Omega)} = \mathcal{O}\left( d^{\frac{m}{2}+c(k,m)}\left(1+2^{k}d^{\frac{k}{2}}M(d)+k!M(d)\right)^{\frac{m}{k}}\left(\delta+\sqrt{\varepsilon}\right)^{\frac{k-m}{k}} \right),
    $$
    asymptotically as $d\to \infty$.
    The implied constant is only related to $k$, $m$, $\|f\|_{L^2(\Omega)}$, and $|f|_{\F_{k+1}(\Omega)}$.
\end{theorem}
\begin{proof} By the minimizing property, we have
    \begin{equation*}
        \begin{split}
            L(f_{\lambda}^{\delta,\varepsilon})\le & L(f)+ \varepsilon
                \le \delta^2 + \lambda |f|_{\F_{k+1}(\Omega)}^2 +  \varepsilon.
        \end{split}
    \end{equation*}
    We thus obtain the bounds
    $$
    \|f^{\delta}-f_{\lambda}^{\delta,\varepsilon}\|_{L^2(\Omega)} \le \delta+ \sqrt{\lambda} |f|_{\F_{k+1}(\Omega)} + \sqrt{ \varepsilon},
    $$
    and
    \begin{equation*}
        \begin{split}
            \|f_{\lambda}^{\delta,\varepsilon}\|_{L^2(\Omega)}
            & \le \|f_{\lambda}^{\delta,\varepsilon}-f^{\delta}\|_{L^{2}(\Omega)}+\|f^{\delta}-f\|_{L^{2}(\Omega)}+\|f\|_{L^{2}(\Omega)}\\
            & \le 2\delta + \sqrt{\lambda}|f|_{\F_{k+1}(\Omega)} + \sqrt{ \varepsilon} + \|f\|_{L^{2}(\Omega)}\\
            & = 3\delta + 2\sqrt{ \varepsilon} + \|f\|_{L^{2}(\Omega)},
        \end{split}
    \end{equation*}
    for $\sqrt{\lambda}|f|_{\F_{k+1}} = \delta+\sqrt{ \varepsilon}$.
    At the same time, we also derive
    $$
    |f^{\delta,\varepsilon}_{\lambda}|_{\F_{k+1}(\Omega)} \le \lambda^{-\frac{1}{2}}\left(\delta + \sqrt{\lambda}|f|_{\F_{k+1}(\Omega)} + \sqrt{\varepsilon} \right) = 2|f|_{\F_{k+1}}.
    $$
    Combining the above inequality, we further obtain
    \begin{align}
        & \|f-f_{\lambda}^{\delta,\varepsilon}\|_{L^{2}(\Omega)}\le 2\delta + \sqrt{\lambda} |f|_{\F_{k+1}(\Omega)} + \sqrt{ \varepsilon}=3\delta + 2\sqrt{ \varepsilon}\le 3(\delta +\sqrt{\varepsilon}), \label{eq:L2_error_RBV} \\
        & \|f_{\lambda}^{\delta,\varepsilon}\|_{L^{2}(\Omega)}+|f_{\lambda}^{\delta,\varepsilon}|_{\F_{k+1}(\Omega)} \le 3\delta + 2\sqrt{ \varepsilon}+\|f\|_{L^{2}(\Omega)}+2|f|_{\F_{k+1}(\Omega)}. \label{eq:L2_error_RBV2}
    \end{align}
    Then according to \ref{enm:sobolev_RBV} in Theorem \ref{thm:sobolev_3spaces} and the above (\ref{eq:L2_error_RBV2}), there holds
    \begin{equation}\label{eq:Hk_error_RBV}
        \begin{split}
            \|f_{\lambda}^{\delta,\varepsilon}-f\|_{H^{k}(\Omega)}\le  &\|f_{\lambda}^{\delta,\varepsilon}\|_{H^k(\Omega)} + \|f\|_{H^{k}(\Omega)}\\
            %& \le   \tilde{c}\left( 1+2^{k}M(d)d^{\frac{k}{2}} \right)\left( \|f_{\lambda}^{\delta,\varepsilon}\|_{L^{2}(\Omega)} + |f_{\lambda}^{\delta,\varepsilon}|_{\F_{k+1}(\Omega)} \right)\\
            %& \quad + \tilde{c}\left( 1+2^{k}M(d)d^{\frac{k}{2}} \right)\left( \|f\|_{L^2(\Omega)} + |f|_{\F_{k+1}(\Omega)}\right)\\ 
            \le & 3\frac{\tilde{c}}{k!}\left(1+2^{k}d^{\frac{k}{2}}M(d)+k!M(d)\right)\\
            &\times \left(\delta+\sqrt{\varepsilon}+\left\|f\right\|_{L^{2}(\Omega)}+\left|f\right|_{\mathcal{F}_{k+1}(\Omega)}\right),\\
            % & \le \tilde{c}\left( 1+2^{k}M(d)d^{\frac{k}{2}} \right) \left( 3\delta + 2\sqrt{ \varepsilon} +2\|f\|_{L^{2}(\Omega)} + 3|f|_{\F_{k+1}(\Omega)}\right)\\
            % & \le 3\tilde{c}\left( 1+2^{k}M(d)d^{\frac{k}{2}} \right)\left( \delta + \sqrt{ \varepsilon} +\|f\|_{L^{2}(\Omega)} + |f|_{\F_{k+1}(\Omega)}\right),
        \end{split}
    \end{equation}
    where the $\tilde{c}:= \tilde{c}(d,k,k)$ is the same as in Lemma \ref{lem:approx_var_L2} and the $M(d)$ is the same as in \eqref{eq:var_RBV}. Finally, the interpolation inequality \eqref{eq:interpolation}, (\ref{eq:L2_error_RBV}), and (\ref{eq:Hk_error_RBV}) yield
    \begin{equation*}
        \begin{split}
            \|f-f_{\lambda}^{\delta,\varepsilon}\|_{H^{m}(\Omega)}
            & \le K(d,k,m)\|f-f_{\lambda}^{\delta,\varepsilon}\|_{H^{k}(\Omega)}^{\frac{m}{k}}\|f-f_{\lambda}^{\delta,\varepsilon}\|_{L^{2}(\Omega)}^{\frac{k-m}{k}}\\
            & \le 3\cdot 2^{\frac{m}{k}}K(d,k,m)\left(\frac{\tilde{c}}{k!}\right)^{\frac{m}{k}}\left(1+2^{k}d^{\frac{k}{2}}M(d)+k!M(d)\right)^{\frac{m}{k}}\\
            &\quad \times \left(\delta + \sqrt{ \varepsilon}\right)^{\frac{k-m}{k}} \left[\left(\delta+\sqrt{\varepsilon}\right)^{\frac{m}{k}} + \left(\left\|f\right\|_{L^{2}(\Omega)} + \left|f\right|_{\mathcal{F}_{k+1}(\Omega)}\right)^{\frac{m}{k}}\right].
        \end{split}
    \end{equation*}
    Noticing that $\tilde{c}\sim d^{\frac{k}{2}}$ when $d\to \infty$ and \eqref{eq:K_dkm}, we end the proof.
\end{proof}

In the above theorem, we keep the factor $(1+2^{k}M(d)+k!M(d))^{\frac{m}{k}}$ for the sake of completeness. According to the lower bound of $M(d)$ provided in Lemma \ref{lem:M_lowerbound}, this factor exceeds a positive constant as $d$ approaches infinity. However, the upper bound of $M(d)$ remains uncertain.

\subsection{Discussion of the above error bounds}
As Theorems \ref{thm:Barron_regularization}-\ref{thm:RadonBV_regularization} demonstrate, all three (semi)norm penalty terms exhibit regularization properties when approximating a function and its derivatives. Their distinct characteristics can be summarized as follows:
The extended Barron norm penalty term includes both the weights and biases of the hidden layer as well as the weights of the output layer. Conversely, the variation norm penalty term contains the weights of the output layer, and the control over the weights and biases of the hidden layer is conducted during the training of the dictionary $\PP_k$. These two types of norms achieve the properties of regularization solutions under the $H^{k}(\Omega)$ norm by controlling these weights and biases. For the most special Radon-BV seminorm penalty term, since low-order polynomials belong to the null space of the operator $\RBV_{k+1}$, we need to additionally analyze the $L^2(\Omega)$ norm of the regularized solution to assist in controlling its $H^{k}(\Omega)$ norm. Fortunately, we achieve the goal by controlling the residual term in the Tikhonov functional (\ref{eq:Tik_RadonBV}).

We make some further comments concerning these error bounds in Theorems \ref{thm:Barron_regularization}-\ref{thm:RadonBV_regularization}.
\begin{enumerate}
  \item %{\color{red}\sout{Although one could potentially design (optimal) networks tailored specifically for an unknown function or a particular derivative,}}
  It is noteworthy that all regularization minimizers presented in Theorems \ref{thm:Barron_regularization}-\ref{thm:RadonBV_regularization} are capable of approximating the function and its derivatives simultaneously. This observation offers significant computational efficiency when designing a neural network with diverse approximating tasks, as it allows for a unified approach that handles multiple objectives concurrently.
  \item To implement Tikhonov regularization with the extended Barron norm penalty or the variation norm penalty, it is necessary to predetermine the number of neurons, $n$, in order to explicitly compute the penalty value during the optimization of regularization functionals. This specific value of $n$ subsequently appears in the error bounds of Theorems \ref{thm:Barron_regularization} and \ref{thm:vs_regularization}, allowing for the adjustment of $n$ to balance the total error with respect to the noise level $\delta$. Contrastingly, the Radon-BV seminorm penalty term does not require prior knowledge of the neuron number $n$. In principle, any two-layer neural network with an arbitrary number of neurons can belong to the Radon-BV space. Therefore, the neuron number $n$ can be arbitrarily large, resulting in an error bound in Theorem \ref{thm:RadonBV_regularization} that is independent of the value of $n$.
  \item
  %{\color{red}\sout{When considering function approximation, the Tikhonov regularization with the extended Barron norm penalty term offers an error bound that is independent of the dimensionality index $d$.}}
  While the other two (semi)norms undoubtedly possess their own merits, such as requiring fewer network parameters or allowing for a wider search space with arbitrary widths, the extended Barron norm stands out for its insensitivity to the dimensionality of the function approximation problem with a finite number of neurons. %Therefore, it seems advisable to prioritize the extended Barron norm as a regularization method, especially when dealing with high-dimensional data.
  \item When considering derivative approximation, all approaches typically result in error bounds that are of H\"older-type dependency on the dimensionality index $d$, which means there is no curse of dimensionality in this particular inverse problem of numerical differentiation.
      However, given the moderate ill-posedness of numerical differentiation, this observation prompts an interesting inquiry: Is the curse of dimensionality a consequence when tackling other severely ill-posed inverse problems? %To delve into this question, the examination of singular values or stability estimates as a function of the dimensionality index for a range of severely ill-posed inverse problems in high dimensions holds significant interest.
      While a recent study in \cite[Sec. 5]{MH24} sheds light on this matter, it primarily only focuses on a moderately ill-posed problem. Further investigations into severely ill-posed problems are warranted to gain a more comprehensive understanding.
\end{enumerate}

\section{Conclusion}
In this paper, we undertake a comprehensive examination of three prevalent (semi)norms encountered in shallow neural networks, delving into their inherent relationships. Specifically, we investigate the linkage between these norms and the standard Sobolev norms, while focusing on the dimensionality index of the domain under consideration. A generic Tikhonov regularization scheme is introduced for approximating functions and their derivatives. A rigorous error-bound analysis in Sobolev norms is conducted, taking into account the influence of the dimensionality index. %It is noteworthy that, even with a unit cubic domain in $\R^d$, the effect of dimensionality cannot be overlooked when approximating derivatives.

For future work, several avenues of investigation present themselves. Firstly, a comprehensive numerical comparison of various regularization schemes should be conducted to validate the differences among them. %A particular focus should be placed on the interplay between the number of neurons and the dimensionality, exploring how these factors influence the performance of the regularization methods.
Secondly, our current research has focused primarily on shallow neural networks; thus, extending these findings to deep neural networks represents an imporant next step. Finally, verifying the curse of dimensionality and investigating the potential application of these concepts to a wider array of inverse problems holds significant interest, as it could potentially find new pathways for insights and solutions within this field.

% \section*{Acknowldegements}
% This work is supported by National Key Research and Development Programs of China (No. 2023YFA1009103), NSFC (No. 11925104) and Science and Technology Commission of Shanghai Municipality (No. 23JC1400501).

\bibliographystyle{siamplain}
\bibliography{Refs_FunDeApp}

\end{document}